\documentclass[11pt]{article}

\usepackage[letterpaper,margin=1.08in]{geometry}
\usepackage[T1]{fontenc}
\usepackage[utf8]{inputenc}
\usepackage{lmodern}
\usepackage{microtype}
\usepackage{amsmath,amssymb,amsthm}
\usepackage{graphicx}
\usepackage{flafter}
\usepackage[authoryear,round]{natbib}
\usepackage{xcolor}
\usepackage{hyperref}
\usepackage{aliascnt}

\hypersetup{
  colorlinks=true,
  linkcolor=blue!55!black,
  citecolor=green!40!black,
  urlcolor=blue!60!black,
  pdfauthor={Zikai Xiong and Robert Freund},
  pdftitle={Symmetry-dependence in Rounding of a Convex Body}
}

\allowdisplaybreaks
\newtheorem{theorem}{Theorem}[section]
\newaliascnt{lemma}{theorem}
\newtheorem{lemma}[lemma]{Lemma}
\aliascntresetthe{lemma}

\newaliascnt{remark}{theorem}
\newtheorem{remark}[remark]{Remark}
\aliascntresetthe{remark}

\usepackage[nameinlink,capitalize,noabbrev]{cleveref}
\crefname{theorem}{Theorem}{Theorems}
\crefname{lemma}{Lemma}{Lemmas}
\crefname{remark}{Remark}{Remarks}
\Crefname{theorem}{Theorem}{Theorems}
\Crefname{lemma}{Lemma}{Lemmas}
\Crefname{remark}{Remark}{Remarks}
\crefformat{equation}{(#2#1#3)}
\Crefformat{equation}{(#2#1#3)}
\crefrangeformat{equation}{(#3#1#4)--(#5#2#6)}
\Crefrangeformat{equation}{(#3#1#4)--(#5#2#6)}
\crefmultiformat{equation}{(#2#1#3)}{ and~(#2#1#3)}{, (#2#1#3)}{, and~(#2#1#3)}
\Crefmultiformat{equation}{(#2#1#3)}{ and~(#2#1#3)}{, (#2#1#3)}{, and~(#2#1#3)}

\title{Symmetry-dependence in Rounding of a Convex Body}
\author{Zikai Xiong\thanks{Department of Industrial Engineering and Management Sciences, Northwestern University,
2145 Sheridan Road, Evanston, IL 60208, USA. \href{mailto:zikai.xiong@northwestern.edu}{zikai.xiong@northwestern.edu}.}
\and Robert M. Freund\thanks{MIT Sloan School of Management, 77 Massachusetts Avenue, Cambridge, MA 02139, USA.
\href{mailto:rfreund@mit.edu}{rfreund@mit.edu}.}}
\date{September 23, 2026}

\begin{document}
\vspace{-10pt}
\maketitle

\vspace{-20pt}
\begin{abstract}
The symmetry measure of a convex body $S\subset\mathbb{R}^n$ is given by:
\begin{equation*}
 \mathop{\mathbf{sym}}\nolimits(S):=\max\{\alpha\ge0:\text{ there exists }x\in S\text{ such that }
 -\alpha(S-x)\subseteq S-x\}\,,
\end{equation*} where such an $x$ is called a Minkowski center. 
We prove that every convex body $S$ admits a $\sqrt{\frac{n}{\mathop{\mathbf{sym}}\nolimits(S)}}$-rounding of $S$, namely, there exists an origin-centered ellipsoid $E$ and a center $c$ 
such that the following rounding holds:
\begin{equation*}
 E\subseteq S-c\subseteq\sqrt{\frac{n}{\mathop{\mathbf{sym}}\nolimits(S)}}\,E\,.
\end{equation*} 
This result was conjectured in 2005 by Belloni and Freund. As special cases, this recovers the known result of an $n$-rounding of $S$ (since it always holds that $\mathop{\mathbf{sym}}\nolimits(S) \ge \tfrac{1}{n}$), and also recovers the known result of a $\sqrt{n}$-rounding when $\mathop{\mathbf{sym}}\nolimits(S) = 1$.  

In the case when $S$ is a polytope given as the convex hull of points, the desired rounding is produced by a regularized minimum-volume covering ellipsoid problem where the regularization is with respect to the Minkowski center. Similarly, when $S$ is a polytope given as the intersection of halfspaces, such a rounding is produced by a regularized maximum-volume inscribed ellipsoid problem. In both of these cases, the rounding can be computed by first solving a linear optimization problem (to compute $\mathop{\mathbf{sym}}\nolimits(S)$ and a Minkowski center), and then solving a convex optimization problem with a logarithmic determinant objective, second-order cone constraints, and one semidefinite cone constraint. 

We also show that the factor $\sqrt{\tfrac{n}{\mathop{\mathbf{sym}}\nolimits(S)}}$ is nearly tight in its dependence on dimension and symmetry. When
$\tfrac{n+1}{1+\mathop{\mathbf{sym}}\nolimits(S)}$ is an integer, we show by explicit construction that the factor $\sqrt{\tfrac{n}{\mathop{\mathbf{sym}}\nolimits(S)}}$ is tight. In the more general case, for every dimension $n$ and every admissible symmetry value, we construct a polytope $S$ for which every rounding factor is at least $\sqrt{\tfrac{2}{3}}\sqrt{\tfrac{n}{\mathop{\mathbf{sym}}\nolimits(S)}}$. 
\end{abstract}



\section{Introduction}

Let $S \subset \mathbb{R}^n$ be a full-dimensional convex body, namely a closed bounded convex set with nonempty interior.  
Let $c\in\mathbb{R}^n$ and $E \subset \mathbb{R}^n$ be an ellipsoid centered at the origin.  For $\rho >0$ we say that $c+E$ is a \emph{$\rho$-rounding} of $S$ if
\begin{equation}\label{eq:rounding-intro}
 E\subseteq S-c\subseteq\rho E\,, 
\end{equation}
and we call $\rho$ the \emph{rounding factor}.   
A $\rho$-rounding of $S$ means that $c + E$ approximates $S$ to within a factor of $\rho$; under the linear map that maps $E$ to the Euclidean unit ball, $S$ is mapped between the Euclidean ball and its $\rho$-scaled copy.
Roundings arise in cutting-plane methods
\citep{GrotschelLovaszSchrijver1993},
geometric random-walk sampling \citep{Lovasz1999},
volume computation \citep{KannanLovaszSimonovits1997},
integer optimization \citep{Lenstra1983},
and the design of efficient gradient methods for linear programming
\citep{Nesterov2008}.
They are also related to ellipsoidal
uncertainty models in robust optimization \citep{BenTalNemirovski1998}.

Polyhedra give rise to several optimization-related roundings.  If $S$ is a bounded polyhedron described by $m$
linear inequalities, the analytic center together with its Dikin ball yields an $(m-1)$-rounding of $S$ \citep{Dikin1967,Sonnevend1986}. \citet{Anstreicher1999} showed that an $O(n)$-rounding of $S$ can be constructed using points at or near the volumetric center. Indeed, there are extended versions of these results that apply to analytic centers of any convex body via the theory of self-concordant functions \citep{NesterovNemirovskii1994}.

For a general convex body $S$, the minimum-volume covering ellipsoid (MVCE, also known as the L\"owner ellipsoid) and the maximum-volume inscribed ellipsoid (MVIE, also known as the John ellipsoid) provide two other classical constructions.
Each yields an $n$-rounding for general $S$, and also a $\sqrt n$-rounding of $S$ if $S$ is centrally symmetric; furthermore these factors are sharp \citep{John1948,Todd2016}.

Algorithms for computing roundings via the two extremal ellipsoid problems MVCE and MVIE have also been the subject of research investigations. 
\citet{KhachiyanTodd1993} showed how to
approximate the MVIE of a polytope described by linear inequalities.
\citet{Khachiyan1996} provided a $(1+\varepsilon)n$-rounding algorithm based on the dual formulation of the MVCE problem. Other computational approaches include \citet{SunFreund2004,ZhangGao2003}. 
Through convex duality, the MVCE problem is also closely connected to $D$-optimal experimental design \citep{Guertuna2010}.

The sharp factors of $n$- and $\sqrt{n}$-roundings for general and centrally symmetric convex bodies (respectively) raise the question of what rounding factor can be attained at
an intermediate degree of symmetry?  To quantify this, define the symmetry of $S$ about $x\in S$ by
\begin{equation*}
 \mathop{\mathbf{sym}}\nolimits(x,S):=\max\left\{\alpha\ge0\,\middle|\,-\alpha(S-x)\subseteq S-x\right\}\,,
\end{equation*}
and the symmetry of $S$ by
\begin{equation*}
 \mathop{\mathbf{sym}}\nolimits(S):=\max_{x\in S}\mathop{\mathbf{sym}}\nolimits(x,S)\,.
\end{equation*}
Thus $\mathop{\mathbf{sym}}\nolimits(x,S)$ measures how far $S-x$ can be reflected through the origin while remaining in $S-x$, and a
point attaining the optimized value $\mathop{\mathbf{sym}}\nolimits(S)$ is called a \emph{Minkowski center}.

For any convex body $S$ it holds that $\mathop{\mathbf{sym}}\nolimits(S) \in [\tfrac1n,1]$; the upper bound follows from the definition of $\mathop{\mathbf{sym}}\nolimits(S)$, and the lower bound is a consequence of the existence of an $n$-rounding of $S$ which implies that the center $c$ of such a rounding must have $\mathop{\mathbf{sym}}\nolimits(c,S) \ge \tfrac{1}{n}$.  Both limits are attained: for any centrally symmetric convex body $S$ (such as the unit ball of a norm) we have $\mathop{\mathbf{sym}}\nolimits(S) =1$, and $\mathop{\mathbf{sym}}\nolimits(S) = \tfrac{1}{n}$ if (and only if) $S$ is any $n$-dimensional simplex  \citep{Minkowski1897,Radon1916,Klee1953}.  \citet{BelloniFreund2008} conjectured that for any convex body $S$ there exists a rounding with
factor $\sqrt{\tfrac{n}{\mathop{\mathbf{sym}}\nolimits(S)}}$:
\begin{equation}\label{eq:conjectured-rounding-factor}
 E\subseteq S-c\subseteq\sqrt{\frac{n}{\mathop{\mathbf{sym}}\nolimits(S)}}\,E\,.
\end{equation}
The quantity $\sqrt{\tfrac{n}{\mathop{\mathbf{sym}}\nolimits(S)}}$ interpolates between the endpoints $n$ (for general convex bodies) and $\sqrt{n}$ (for symmetric convex bodies).  Our main result is a proof of this rounding guarantee, which is formally stated below in \Cref{thm:global-rounding}. 

The symmetry measure $\mathop{\mathbf{sym}}\nolimits(x,S)$ occurs elsewhere in optimization. For example in conic optimization, it enters
interior-point complexity bounds \citep{BelloniFreund2009}
and is related to condition measures and well-posedness
\citep{PenaRoshchina2020}.
In robust and adaptive optimization, $\mathop{\mathbf{sym}}\nolimits(x,S)$ and $\mathop{\mathbf{sym}}\nolimits(S)$ enter approximation guarantees
\citep{BertsimasGoyalSun2011}.
Computing a Minkowski center has also been formulated as a
robust optimization problem and has been shown computationally to improve
the convergence of hit-and-run and cutting-plane methods
\citep{denHertogPauphiletSoali2024}.

Existing symmetry-dependent rounding guarantees
are expressed in terms of the symmetry at a prescribed center.  For an arbitrary $x\in\operatorname{int} S$, \citet{BelloniFreund2008} proved the existence of a
$\tfrac{\sqrt{n}}{\mathop{\mathbf{sym}}\nolimits(x,S)}$-rounding centered at $x$ and obtained the sharper factor
$\sqrt{\tfrac{n}{\mathop{\mathbf{sym}}\nolimits(x_L,S)}}$ at the L\"owner center $x_L$.  At the John center $x_J$, the analogous
$\sqrt{\tfrac{n}{\mathop{\mathbf{sym}}\nolimits(x_J,S)}}$-rounding follows from \citet{BrandenbergKoenig2015}.  All of these results use the local symmetry $\mathop{\mathbf{sym}}\nolimits(x,S)$ and not the global (and optimized) symmetry $\mathop{\mathbf{sym}}\nolimits(S)$.
Indeed, the symmetry at the
John center can be smaller than $\mathop{\mathbf{sym}}\nolimits(S)$ by a factor of order $n$, and even a convex body with global symmetry greater
than $\tfrac12$ may require a rounding with factor $n$ when using the MVIE
\citep[Example~A.1]{BrandenbergGrundbacher2025}.  

Our constructive proof of the conjectured bound \eqref{eq:conjectured-rounding-factor} modifies the MVCE problem for polytopes by adding a symmetry-dependent Minkowski center regularization term.  In the case when $S$ is a polytope (represented either as the convex hull of points or as the intersection of halfspaces), the desired rounding can be computed by first solving a linear optimization problem (to compute $\mathop{\mathbf{sym}}\nolimits(S)$ and a Minkowski center), and then solving a convex optimization problem formulation of either the associated MVCE or MVIE problem depending on the polytope representation. 

We also show that the bound \eqref{eq:conjectured-rounding-factor} is sharp up to a universal constant for every $n$ and every admissible symmetry value $\alpha$, and is exactly sharp for an infinite class of pairs of $n$ and $\alpha$.

\Cref{fig:four-roundings} shows a $2$-dimensional polytope $K$ (in blue) along with (a) the MVCE ellipsoidal rounding, (b) the MVIE rounding, (c) our center-regularized MVCE rounding, and (d) our center-regularized MVIE
rounding. Numerical details are provided in Appendix \ref{app:numerical-example}.  For this polytope, both classical
volume-extremal roundings only attain the dimension-dependent factor $2$, whereas both of our constructions yield
rounding factors less than the bound $\sqrt{\tfrac{n}{\mathop{\mathbf{sym}}\nolimits(K)}} \approx 1.6816$ and select centers visibly closer to the Minkowski center.

\begin{figure}[!htbp]
\centering
\begin{minipage}[t]{0.48\textwidth}
 \centering
 \includegraphics[width=1\linewidth]{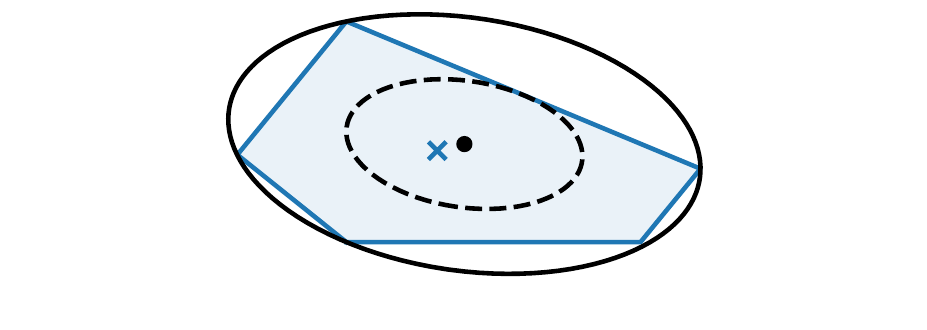}\\[-0.4ex]
 \small (a) Minimum-volume covering ellipsoid.
\end{minipage}\begin{minipage}[t]{0.48\textwidth}
 \centering
 \includegraphics[width=1\linewidth]{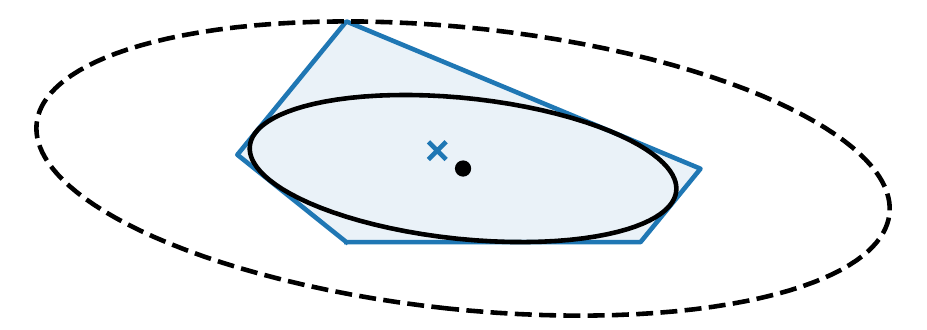}\\[-0.4ex]
 \small (b) Maximum-volume inscribed ellipsoid.
\end{minipage} 
\begin{minipage}[t]{0.48\textwidth}
 \centering
 \includegraphics[width=1\linewidth]{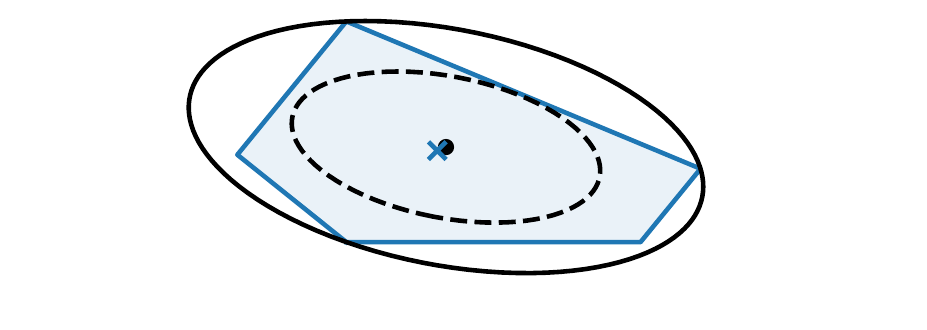}\\[-0.4ex]
 \small (c) The center-regularized minimum-volume covering ellipsoid from 
 \Cref{lm:rounding-w-representation}.
\end{minipage}\begin{minipage}[t]{0.48\textwidth}
 \centering
 \includegraphics[width=1\linewidth]{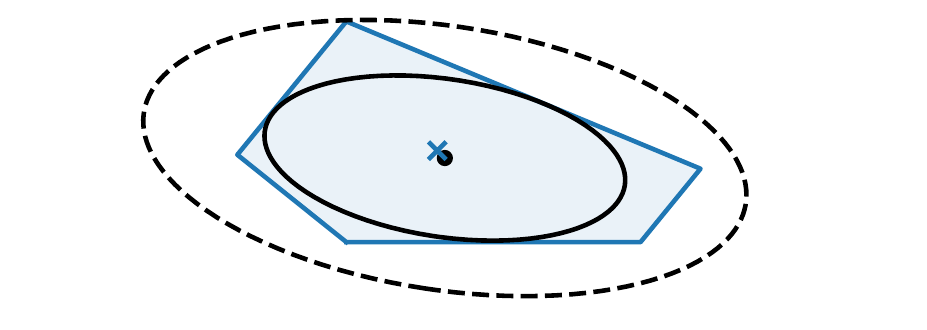}\\[-0.4ex]
 \small (d) The center-regularized maximum-volume inscribed ellipsoid from 
 \Cref{lm:john-rounding}.
\end{minipage}
\caption{Four ellipsoidal roundings of the polytope $K$.
The blue cross denotes the Minkowski center of $K$. The black dots denote ellipsoid centers.}
\label{fig:four-roundings}
\end{figure}

For $Q\succ0$, define the origin-centered ellipsoid
\begin{equation*}
 {E}_Q:=\left\{x\in\mathbb{R}^n\,\middle|\,x^\top Qx\le1\right\}\,.
\end{equation*}
Our main result is the proof of the following symmetry-dependent rounding bound for every convex body.

\begin{theorem}\label{thm:global-rounding}
Let $S\subset\mathbb{R}^n$ be a convex body.  Then there exists a rounding of $S$ with factor $\sqrt{\tfrac{n}{\mathop{\mathbf{sym}}\nolimits(S)}}$, namely there exist $c \in \mathbb{R}^n$ and a matrix $Q\succ0$ such that
\begin{equation}
 {E}_Q
 \subseteq S-c \subseteq
 \sqrt{\frac{n}{\mathop{\mathbf{sym}}\nolimits(S)}}\, {E}_Q\,.
\end{equation}
\end{theorem}

In \Cref{sec:v-representation} we treat the case when $S$ is a polytope. We introduce a center-regularized version of the MVCE problem, we derive an equivalent convex formulation and its
dual, and we prove that the resulting ellipsoid yields the desired rounding stated in \Cref{thm:global-rounding}.
In \Cref{sec:polytope-to-convex-body} we prove the general case of \Cref{thm:global-rounding} from the polyhedral case by approximating $S$ by a sequence of polytopes and taking limits.
In \Cref{sec:h-representation} we further address the case when $S$ is a polytope but is described as the intersection of finitely many halfspaces.  In this representation, we present a center-regularized version of the MVIE problem and use this problem to (re-)prove the rounding guarantee of \Cref{thm:global-rounding}. 
In \Cref{sec:sharpness-family} we show that the rounding factor in \eqref{eq:conjectured-rounding-factor} is nearly tight in its dependence on dimension and symmetry.  
When $\tfrac{n+1}{1+\mathop{\mathbf{sym}}\nolimits(S)}$ is an integer, we show by explicit construction that the factor
$\sqrt{\tfrac{n}{\mathop{\mathbf{sym}}\nolimits(S)}}$ is tight.
In the more general case, for every
$n\ge1$ and admissible $\mathop{\mathbf{sym}}\nolimits(S)$, we construct a polytope for which every rounding factor is at least
$\sqrt{\tfrac{2n}{3\mathop{\mathbf{sym}}\nolimits(S)}}$.  

Concurrent work of \citet{BrandenbergGonzalezMerinoGrundbacher2026} in the context of Banach--Mazur distance establishes a related
nonconcentric ellipsoidal containment guarantee. Stating their result in our notation, they
prove that for every convex body $S$ there exists an origin-centered ellipsoid $E$ and possibly distinct centers $c$, $d$  
such that
$c+E\subseteq S\subseteq
    d+\sqrt{\tfrac{n}{\mathop{\mathbf{sym}}\nolimits(S)}}E$. 
    In their construction, the outer ellipsoid is the MVCE of $S$, and
the inner ellipsoid is a translated and scaled copy of the outer ellipsoid.
They also present an example showing that the translation cannot in
general be removed while retaining the MVCE as the outer ellipsoid.
Nevertheless, their result implies a $2\sqrt{\tfrac{n}{\mathop{\mathbf{sym}}\nolimits(S)}}$-rounding in the sense of
\eqref{eq:rounding-intro}.

\paragraph{Notation.}
For $d\ge1$, let $\mathbb{R}_+^d$ denote the nonnegative orthant, let $e$ denote the vector of all ones with dimension
determined by context, and let $I_d$ denote the $d\times d$ identity matrix.  The symbols
$\langle\cdot,\cdot\rangle$ and $\lVert\cdot\rVert_2$ denote the Euclidean inner product and norm, respectively.  For $w\in\mathbb{R}^d$,
$\operatorname{Diag}(w)$ is the diagonal matrix whose diagonal is $w$.  The relations $Q\succeq0$ and $Q\succ0$ mean that $Q$ is positive semidefinite and positive definite,
respectively.  For $Q\succ0$, $Q^{\tfrac12}$ denotes its unique positive definite square root.  For
$X\succ0$, we use the notation $X^{-2}:=(X^2)^{-1}$.

For a set $\Omega\subseteq\mathbb{R}^n$, a vector $c\in\mathbb{R}^n$, and a scalar $t$, write
$\operatorname{int}\Omega$ for the interior of $\Omega$, $\Omega-c:=\{x-c:x\in\Omega\}$, and $t\Omega:=\{tx:x\in\Omega\}$.  The operator $\mathop{\mathbf{conv}}\nolimits$ denotes the convex hull, and $\mathop{\mathbf{sym}}\nolimits$ denotes the measure of symmetry defined above.  We write $B_2^n$ for the Euclidean unit ball and $\operatorname{vol}_n$ for
$n$-dimensional volume.  A \emph{convex body} is a compact convex subset of $\mathbb{R}^n$ with nonempty interior.

\section{Center-regularized MVCE for a polytope}\label{sec:v-representation}

In this section we prove \Cref{thm:global-rounding} in the case when $S$ is a polytope. We construct a center-regularized version of the MVCE optimization problem, and we show that a solution to this optimization problem achieves the desired $\sqrt{\frac{n}{\mathop{\mathbf{sym}}\nolimits(S)}}$-rounding.

As every point in a polytope can be expressed as a weighting of finitely many points, we consider the following \textit{$W$-representation} of $S$:
\begin{equation}\label{eq:global-polytope-setup}
 S=\mathop{\mathbf{conv}}\nolimits\{a_1,\ldots,a_m\}\,,
\end{equation}
where $a_1,\ldots,a_m$ are the listed points.  Since $S$ is assumed to be a convex body, we further assume without loss of generality that the origin is in the interior of $S$ (which can be accomplished by choosing any interior point of $S$ and translating to the origin).   
While the representation in \eqref{eq:global-polytope-setup} is commonly referred
to as a $V$-representation, we prefer to call it a $W$-representation, with $W$ standing for ``weighting of points'' since we do not need to assume that each of the $a_i$ is a vertex.

First of all, we recall the ordinary MVCE problem. 
For an ellipsoid with center $c$ and shape matrix $Q$, denoted by $c+{E}_Q$,  its volume is
$\operatorname{vol}_n(c+{E}_Q)=\operatorname{vol}_n(B_2^n)(\det Q)^{-\tfrac12}$ \citep{GrotschelLovaszSchrijver1993}.  Because an ellipsoid is
convex, it contains $S$ if and only if it contains $a_1, \ldots, a_m$.  Thus the MVCE solves the optimization problem
\begin{equation}\label{eq:free-loewner-primal}
\begin{aligned}
 \operatorname*{minimize}_{Q,c} \ \quad & -\log\det Q \\
 \text{subject to}\quad
   &(a_i-c)^\top Q(a_i-c)\le1\,,
   \qquad i=1,\ldots,m\,,\\
   &Q\succ0,\qquad c\in\mathbb{R}^n\,.
\end{aligned}
\end{equation}
This formulation is standard \citep{Todd2016}.

To motivate our center-regularized version of the MVCE problem, suppose for the sake of argument that a point $x^h$ with a high symmetry value $\mathop{\mathbf{sym}}\nolimits(x^h, S)$ is known (e.g., $x^h$ could be a Minkowski center of $S$), and that $S$ has been translated by $x^h$ so that $\mathop{\mathbf{sym}}\nolimits(0,S)$ is high.  Then the origin lies ``deep'' inside $S$.
This suggests regularizing the ellipsoid center of the MVCE toward the origin while allowing the ellipsoid center and its shape to adjust jointly.  For a regularization parameter $\lambda\ge0$, we therefore consider the following
center-regularized MVCE problem:  
\begin{equation}\tag{$\mathrm{P}_\lambda$}\label{eq:center-regularized-loewner}
\stepcounter{equation}
\begin{aligned}
 \operatorname*{minimize}_{Q,c} \ \quad & -\log\det Q \\
 \text{subject to}\quad
   &(a_i-c)^\top Q(a_i-c)
     +\lambda c^\top Qc\le1\,,
     \qquad i=1,\ldots,m\,,\\
   &Q\succ0,\qquad c\in\mathbb{R}^n\,.
\end{aligned}
\end{equation}
Compared with \eqref{eq:free-loewner-primal}, the additional term in each constraint discourages the center
$c$ from moving away from the origin.  For every feasible $(Q,c)$ it still holds that $S-c\subseteq{E}_Q$, and the
objective function still minimizes the volume of ${E}_Q$.  At $\lambda=0$ \eqref{eq:center-regularized-loewner} specializes to \eqref{eq:free-loewner-primal}, while at $\lambda=+\infty$ \eqref{eq:center-regularized-loewner} is equivalent to fixing $c=0$.
 
It will be useful to set the regularization parameter $\lambda$ in accordance with the symmetry of the origin:
\begin{equation}\label{eq:alpha-lambda}
    \alpha:=\mathop{\mathbf{sym}}\nolimits(0,S)\,,
    \qquad
    \lambda_\alpha:=\frac{1+\alpha}{1-\alpha}\,,
\end{equation}
where in the case when $\alpha = 1$ we define $\lambda_1:= +\infty$. Geometrically, a larger value of $\alpha$ means that the origin is more centrally located in $S$. We therefore treat higher symmetry $\alpha$ as
stronger geometric information that the origin should be a good ellipsoidal rounding center, and so penalize deviations from the origin more heavily.  When $S$ is centrally symmetric about the origin,
$\lambda_\alpha=+\infty$, and the ellipsoid center $c$ is fixed at the origin in \eqref{eq:center-regularized-loewner}. We now present the following rounding results regarding optimal solutions of \eqref{eq:center-regularized-loewner}, the proof of which will be the subject of most of this section.

\begin{lemma}\label{lm:rounding-w-representation}
Let $S$ be a polytope defined as in \cref{eq:global-polytope-setup} and suppose that $0 \in \operatorname{int} S$, and let $\alpha$ and
$\lambda_\alpha$ be defined as in \cref{eq:alpha-lambda}.  The center-regularized problem
\eqref{eq:center-regularized-loewner} with $\lambda=\lambda_\alpha$ attains its optimal value.  Every optimal
solution $(Q^\star,c^\star)$ satisfies 
\begin{equation*}
 \sqrt{\frac\alpha n}\,{E}_{Q^\star}
 \subseteq S-c^\star
 \subseteq{E}_{Q^\star}\,.
\end{equation*}
Moreover, the center $c^\star$ satisfies $c^\star \in \tfrac{1-\alpha}{2}S$. Furthermore, when $\alpha < 1$, the ellipsoid of the outer containment can be tightened to
$ \sqrt{1-\lambda_\alpha(c^\star)^\top Q^\star c^\star}\,{E}_{Q^\star}$.  

\end{lemma}

Consider the case when $S$ has been translated so that the origin is a Minkowski center (so that $\alpha = \mathop{\mathbf{sym}}\nolimits(S)$). Then applying \Cref{lm:rounding-w-representation} yields a 
$\sqrt{\frac{n}{\mathop{\mathbf{sym}}\nolimits(S)}}$-rounding of $S$ and thus proves \Cref{thm:global-rounding} when $S$ is a polytope. 
\Cref{lm:rounding-w-representation} is more general than \Cref{thm:global-rounding} because it allows the origin to be {\em any} interior point of $S$, not necessarily a Minkowski center. Indeed, if $\mathop{\mathbf{sym}}\nolimits(0,S) > \frac{1}{n}$, the center-regularized MVCE problem yields a $\sqrt{\frac{n}{\mathop{\mathbf{sym}}\nolimits(0,S)}}$-rounding of $S$ whose rounding guarantee is strictly better than the general $n$-rounding guarantee of the MVCE.

Note that the optimized center $c^\star$ of the center-regularized MVCE might not be a Minkowski center, nor the origin, nor the center of the ordinary MVCE. \citet{BelloniFreund2008} show that the MVCE center $x_L$ yields a $\sqrt{\tfrac{n}{\mathop{\mathbf{sym}}\nolimits(x_L,S)}}$-rounding, whereas \Cref{lm:rounding-w-representation} gives a $\sqrt{\tfrac{n}{\mathop{\mathbf{sym}}\nolimits(0,S)}}$-rounding of $S$ centered at $c^\star$ where notably, $c^\star$ need not equal $0$. Intuitively, the greater the symmetry of the origin is, the closer $c^\star$ will be to the origin, which is embodied in the property that $c^\star \in \tfrac{1-\alpha}{2}S$.

The center-regularized MVCE is a generalization of the MVCE. If $0$ is a Minkowski center of $S$ and $\lambda$ is chosen as in \eqref{eq:alpha-lambda}, then the center-regularized MVCE recovers the MVCE at the lower and upper extrema of $\mathop{\mathbf{sym}}\nolimits(S)$. When $\mathop{\mathbf{sym}}\nolimits(S)=1$, then $S$ is centrally symmetric about the origin (which is the Minkowski center). When $\mathop{\mathbf{sym}}\nolimits(S)=\frac{1}{n}$, $S$ must be a simplex, whose Minkowski center (the origin) is also the center of the MVCE. In both cases, the center of the MVCE must be the origin. Thus, the optimal solution  $(Q^\star,c^\star)$ of \eqref{eq:free-loewner-primal} is still feasible for \eqref{eq:center-regularized-loewner} because $c^\star=0$. This means $(Q^\star,c^\star)$ is also an optimal solution of \eqref{eq:center-regularized-loewner} for both extrema of $\mathop{\mathbf{sym}}\nolimits(S)$. Furthermore, since every feasible solution of \eqref{eq:center-regularized-loewner} is also feasible for \eqref{eq:free-loewner-primal}, the optimal solution of \eqref{eq:center-regularized-loewner} must be $(Q^\star,c^\star)$ as well for both extrema of $\mathop{\mathbf{sym}}\nolimits(S)$. Additionally, \Cref{lm:rounding-w-representation} also recovers the classical MVCE rounding results for both general polytopes as well as centrally symmetric polytopes.

The center-regularized problem \eqref{eq:center-regularized-loewner} can be transformed into a convex problem by a change of variables.
For $0\le\lambda<+\infty$, make the change of variables
\begin{equation}\label{eq:square-root-variables}
 M:=Q^{\tfrac12}\,,
 \quad\text{and}\quad
 z:=Mc\,,
\end{equation}
which is the standard convexifying transformation used for the MVCE as in \citet{SunFreund2004}. 
Under this change of variables, problem \eqref{eq:center-regularized-loewner} is equivalent to the convex problem
\begin{equation}\tag{$\widetilde{\mathrm P}_\lambda$}\label{eq:transformed-primal}
\stepcounter{equation}
\begin{aligned}
 \operatorname*{minimize}_{M,z}\quad &-2\log\det M\\
 \text{subject to}\quad
 &\lVert Ma_i-z\rVert_2^2+\lambda\lVert z\rVert_2^2\le1\,,
 \qquad i=1,\ldots,m\,,\\
 &M \succ 0 \,,
 \qquad z\in\mathbb{R}^n\,.
\end{aligned}
\end{equation} 
From any solution $(M,z)$ of \Cref{eq:transformed-primal}, the original variables are recovered using $Q=M^2$ and $c=M^{-1}z$.
At the endpoint $\lambda=+\infty$, one sets $z=0$ in \eqref{eq:transformed-primal}, leaving $M$ as the only free variable. 

\subsection{Computing the desired rounding when $S$ is a polytope in $W$-representation}\label{compute paragraph}

In the case when $S$ is a polytope given as the convex hull of points (a $W$-representation of $S$), \Cref{lm:rounding-w-representation} points the way to an efficient scheme for computing the desired $\sqrt{\tfrac{n}{\mathop{\mathbf{sym}}\nolimits(S)}}$-rounding of $S$. One first computes $\mathop{\mathbf{sym}}\nolimits(S)$ and a Minkowski center $x^s$ of $S$.  As shown in \cite{BelloniFreund2008}, this can be accomplished by solving a linear optimization problem in $m^2$ nonnegative variables and $mn + m$ equality constraints. One then translates $S$ by $x^s$ so that the origin is a Minkowski center of $S$ and sets $\alpha := \mathop{\mathbf{sym}}\nolimits(S)$ and $\lambda = \lambda_\alpha := \tfrac{1+\alpha}{1-\alpha}$.  Next solve \eqref{eq:transformed-primal} to obtain an optimal solution $M^\star,z^\star$, and notice that \eqref{eq:transformed-primal} is a convex optimization problem with a logarithmic determinant objective, second-order cone constraints, and one semidefinite cone constraint, and so is amenable to solution by standard interior-point methods, see for example \citep{VandenbergheBoydWu1998}. Finally, set $Q^\star=(M^\star)^2$ and $c^\star = (M^\star)^{-1}z^\star$ which now furnishes the desired rounding of the translation of $S$, and therefore $c^\star \leftarrow c^\star + x^s$ together with $Q^\star$ furnish the desired rounding of the original $S$.

\ 

\noindent The remainder of Section \ref{sec:v-representation} contains necessary intermediary results and then the proof of  \Cref{lm:rounding-w-representation}.  

\subsection{Lagrange dual and optimality conditions of \texorpdfstring{\Cref{eq:transformed-primal}}{the transformed primal problem}}

For notational convenience we collect the listed points in the column-stacked matrix:
\begin{equation*}
 A:=[a_1,\ldots,a_m]\in\mathbb{R}^{n\times m}\,.
\end{equation*}
Let $0\le\lambda<+\infty$, and let $e\in\mathbb{R}^m$ denote the vector of all ones.  Assigning a multiplier
$u_i\ge0$ to the $i$th constraint and using standard
Lagrangian duality yields the following dual problem of \eqref{eq:transformed-primal}:
\begin{equation}\tag{$\mathrm D_\lambda$}\label{eq:lagrange-dual}
\stepcounter{equation}
\begin{aligned}
\operatorname*{maximize}_{u\in\mathbb{R}_+^m}\quad
&n+\log\det\left(
A\operatorname{Diag}(u)A^\top
-\frac{(Au)(Au)^\top}{(1+\lambda)e^\top u}
\right)-e^\top u\\
\text{subject to}\quad
&e^\top u>0\,,\\
&A\operatorname{Diag}(u)A^\top
-\frac{(Au)(Au)^\top}{(1+\lambda)e^\top u}\succ0\,, 
\end{aligned}
\end{equation}
which is a minor extension of the MVCE dual formulations in \citet{SunFreund2004,Guertuna2010}. The primal problem \eqref{eq:transformed-primal} attains its optimal value because its sublevel sets are compact.
In addition, the Slater condition holds since sufficiently small positive multiples of $(I,0)$ are strictly feasible for
\eqref{eq:transformed-primal}, whereby strong duality and dual attainment both hold, and the Karush--Kuhn--Tucker optimality conditions are both necessary and sufficient.

Accordingly, $(M,z)$ and the dual multiplier $u$ are optimal for \eqref{eq:transformed-primal} if and only if the following are satisfied:
\begingroup
\begin{flalign}
 &\text{Primal feasibility:}\quad
 M\succ0\,,\quad z\in\mathbb{R}^n\,,\quad
   \lVert Ma_i-z\rVert_2^2+\lambda\lVert z\rVert_2^2\le1
   \quad\text{for }i=1,\ldots,m\,, &&\label{eq:optimality-primal-feasibility}\\
 &\text{Dual feasibility:}\quad
 u\in\mathbb{R}_+^m\,,\quad e^\top u>0\,,\quad
   A\operatorname{Diag}(u)A^\top
   -\frac{(Au)(Au)^\top}{(1+\lambda)e^\top u}\succ0\,, &&\label{eq:optimality-dual-feasibility}\\
 &\text{Stationarity:}\quad
 z=\frac{MAu}{(1+\lambda)e^\top u}\,,\qquad
   M^2=\left(
   A\operatorname{Diag}(u)A^\top
   -\frac{(Au)(Au)^\top}{(1+\lambda)e^\top u}
   \right)^{-1}\,, &&\label{eq:optimality-stationarity}\\
 &\text{Complementary slackness:}\quad
 u_i\bigl(\lVert Ma_i-z\rVert_2^2
   +\lambda\lVert z\rVert_2^2-1\bigr)=0
   \quad\text{for }i=1,\ldots,m\,. &&\label{eq:optimality-complementary-slackness}
\end{flalign}
\endgroup
At $\lambda=+\infty$, the same
conditions apply after omitting $z$ and all terms containing $z$, and replacing the matrix $A\operatorname{Diag}(u)A^\top
   -\tfrac{(Au)(Au)^\top}{(1+\lambda)e^\top u}$ in
\cref{eq:optimality-dual-feasibility,eq:optimality-stationarity} by
$A\operatorname{Diag}(u)A^\top$.

Similar to the MVCE problem, an immediate consequence of the optimality conditions is that $\sum_i u^\star_i$ is not free.  Let $(M^\star,z^\star,u^\star)$ be an optimal primal-dual triple.  Summing the complementary-slackness equations
\eqref{eq:optimality-complementary-slackness} and then using stationarity \eqref{eq:optimality-stationarity} yields
\begin{align}
 0
 &=\sum_{i=1}^m u_i^\star
 \bigl(
 \lVert M^\star a_i-z^\star\rVert_2^2
 +\lambda\lVert z^\star\rVert_2^2-1
 \bigr)  =\operatorname{tr}\left(M^\star\left[
 A\operatorname{Diag}(u^\star)A^\top
 -\frac{(Au^\star)(Au^\star)^\top}{(1+\lambda)e^\top u^\star}
 \right]M^\star\right)-e^\top u^\star \notag\\
 &=\operatorname{tr}\left(\left[
 A\operatorname{Diag}(u^\star)A^\top
 -\frac{(Au^\star)(Au^\star)^\top}{(1+\lambda)e^\top u^\star}
 \right](M^\star)^2\right)-e^\top u^\star
 =n-e^\top u^\star\,.
\end{align}
Hence any optimal $u^\star$ must satisfy $e^\top u^\star=n$.  The same calculation proves this identity
at $\lambda=+\infty$.  An analogous identity holds for the MVCE problem \citep{SunFreund2004}.

Thus every dual optimizer lies on the slice $e^\top u=n$.  For $u$ on this slice, write
\begin{equation*}
 \tilde u:=\frac un  \qquad \text{and}\qquad
 \widetilde U:=\operatorname{Diag}(\tilde u)\,,
\end{equation*}
and then $e^\top \tilde{u}  = 1$.
Since every dual optimizer must satisfy $e^\top \tilde{u}  = 1$, restricting the dual in this way leaves its optima unchanged and so we consider the ``reduced dual'' problem:
\begin{equation}\tag{$\mathrm D_\lambda^{\rm red}$}\label{eq:reduced-dual}
\stepcounter{equation}
\begin{aligned}
 \operatorname*{maximize}_{\tilde u\in\mathbb{R}_+^m}\quad
 &n\log n+\log\det\left(
 A\widetilde U A^\top-\frac1{1+\lambda}(A\tilde u)(A\tilde u)^\top
 \right)\\
 \text{subject to}\quad
 &e^\top\tilde u=1\, , \quad
 A\widetilde U A^\top-\frac1{1+\lambda}(A\tilde u)(A\tilde u)^\top\succ0\,.
\end{aligned}
\end{equation}
We use the convention $
 \tfrac1{1+\infty}:=0$ to define $(\mathrm D_\infty^{\rm red})$ for consistency. 
Let us set the notation:
\begin{equation*}
 Q_{\tilde u,\lambda}
 :=A\widetilde U A^\top-\frac1{1+\lambda}(A\tilde u)(A\tilde u)^\top\,
\end{equation*}
for $0\le\lambda\le+\infty$.
Then a pair $(M,z)$ and a
vector $\tilde u$ are optimal for \eqref{eq:transformed-primal} and \eqref{eq:reduced-dual} if and only if
$(M,z)$ and $u=n\tilde u$ satisfy the optimality conditions
\eqref{eq:optimality-primal-feasibility}--\eqref{eq:optimality-complementary-slackness}.  In particular, the
stationarity relations \eqref{eq:optimality-stationarity} become
\begin{equation*}
 z=\frac1{1+\lambda}MA\tilde u\,, \qquad
 M^2=\bigl(nQ_{\tilde u,\lambda}\bigr)^{-1} \,.
\end{equation*}
Using $Q=M^2$ and $c=M^{-1}z$, these relations become
\begin{equation}\label{eq:reduced-dual-stationarity-2}
 c=\frac1{1+\lambda}A\tilde u\,, \qquad
 Q^{-1}=nQ_{\tilde u,\lambda}\,.
\end{equation}
At $\lambda=+\infty$, the same statement holds after
omitting $z$, $c$, and all terms where they appear.

\begin{remark}[Connection with $D$-optimal design]

The duality between MVCE problems and
$D$-optimal experimental design is classical
\citep{titterington1975optimal,Guertuna2010}.
For the origin-centered MVCE problem, the dual is the standard
$D$-optimal-design problem on the points $a_1,\ldots,a_m$,
with objective $\det(A\widetilde U A^\top)$.
For the free-center MVCE problem, after the affine lifting $\widehat a_i :=(a_i^\top,1)^\top$, the dual problem is also a standard $D$-optimal-design problem on the lifted
points $\widehat a_1,\ldots,\widehat a_m$.

More generally, let $\widehat A$ denote the column-stacked matrix of $\widehat a_1,\ldots,\widehat a_m$.  Then, for finite $\lambda$, it holds that $\det\!\left(
    \widehat A\widetilde U\widehat A^\top
    +\lambda e_{n+1}e_{n+1}^\top
\right)=(1+\lambda)\det Q_{\tilde u,\lambda}$ (where $e_{n+1}$ is the unit vector in $\mathbb{R}^{n+1}$ with $1$ in its last component).
Thus $(D_\lambda^{\mathrm{red}})$ can be viewed as a modified
$D$-optimal-design problem with an additional fixed information term in the last lifted coordinate.   In this way $(D_\lambda^{\mathrm{red}})$ interpolates between
the $D$-optimal design on the lifted points at $\lambda=0$
and the $D$-optimal design on the original points as
$\lambda\to+\infty$.

The reduced dual $(D_\lambda^{\mathrm{red}})$ also has nice properties for algorithms.  For every
finite $\lambda$, $(D_\lambda^{\mathrm{red}})$ falls within the framework of
log-homogeneous-barrier problems of \citet{dvurechensky2023generalized,zhao2023analysis,zhao2026new}, so the
Frank--Wolfe method applies directly. The earlier
Frank--Wolfe method of \citet{AhipasaogluSunTodd2008} also applies to
the $D$-optimal-design problems at the two extrema of $\lambda$.
\end{remark}

\subsection{Outer and inner ellipsoids from primal and dual feasibility}

In this subsection we present outer and inner ellipsoidal containments obtained from \eqref{eq:center-regularized-loewner} and \eqref{eq:lagrange-dual}.

\begin{lemma}\label{lem:outer-containment}
Let $0\le\lambda<+\infty$.  Every feasible $(Q,c)$ of \eqref{eq:center-regularized-loewner} satisfies
\begin{equation*}
 S-c
 \subseteq\sqrt{1-\lambda c^\top Qc}\,{E}_Q
 \subseteq{E}_Q\,.
\end{equation*}
In the case of $\lambda=+\infty$, the center $c$ is fixed at zero and $S\subseteq{E}_Q$.
\end{lemma}

\begin{proof}
For finite $\lambda$, primal feasibility yields
$(a_i-c)^\top Q(a_i-c)\le1-\lambda c^\top Qc$ for every $i$.  The right-hand side is nonnegative because
the left-hand side is nonnegative, and it is at most $1$ because $\lambda c^\top Qc\ge0$.  Hence the middle
ellipsoid contains every $a_i-c$ and therefore their convex hull which is $S-c$.  Also the middle ellipsoid is itself contained in ${E}_Q$. For $\lambda=+\infty$, the center is fixed at zero and the constraints directly yield
$a_i\in{E}_Q$ for every $i$, hence $S\subseteq{E}_Q$.
\end{proof}

As in \eqref{eq:alpha-lambda} let us assign $\alpha:=\mathop{\mathbf{sym}}\nolimits(0,S)$ and $\lambda:=\lambda_\alpha$, from which it follows that
\begin{equation}\label{eq:alpha-lambda-coefficient}
 \frac1{1+\lambda_\alpha}=\frac{1-\alpha}{2}\,.
\end{equation}
\begin{lemma}\label{lem:dual-inner}
Let $\tilde u$ be feasible for $(\mathrm D_{\lambda_\alpha}^{\rm red})$, and assign
$c:=\tfrac{1-\alpha}{2}A\tilde u$. 
Then $Q_{\tilde u,\lambda_\alpha}\succ0$.  Write
$\widehat Q:=Q_{\tilde u,\lambda_\alpha}$.  We have
\begin{equation}\label{eq:dual-feasible-inner-containment}
 \sqrt\alpha\,{E}_{\widehat Q^{-1}}
 =\left\{x\in\mathbb{R}^n:\ x^\top \widehat Q^{-1}x\le\alpha\right\}
 \subseteq S-c\,.
\end{equation}
Moreover, the center $c$ satisfies $c \in \tfrac{1-\alpha}{2}S$.
\end{lemma}  

\begin{proof}
Since $\tilde u\ge0$ and $e^\top\tilde u=1$, we have $A\tilde u\in S$.  Consequently,
$c=\tfrac{1-\alpha}{2}A\tilde u\in\tfrac{1-\alpha}{2}S$.  

Feasibility for
$(\mathrm D_{\lambda_\alpha}^{\rm red})$ implies $\widehat Q\succ0$.  By the definition of $\widehat Q$ and
\eqref{eq:alpha-lambda-coefficient}, we also have
$\widehat Q=A\widetilde U A^\top-\tfrac{1-\alpha}{2}(A\tilde u)(A\tilde u)^\top$.  To prove
\eqref{eq:dual-feasible-inner-containment}, it follows from a separating hyperplane argument that it is sufficient to show that for every $y\in\mathbb{R}^n$ the following inequality holds:
\begin{equation}\label{eq:inner-directional-goal}
 \max_{x\in\sqrt\alpha\,{E}_{\widehat Q^{-1}}}y^\top x
 \le \max_{x\in S-c}y^\top x\,.
\end{equation}
The case $y=0$ is immediate.  For $y\ne0$, the quantity $\max_i y^\top a_i$ is positive because
$0\in\operatorname{int} S$.  Both sides of \eqref{eq:inner-directional-goal} are positively homogeneous in $y$, so, after
replacing $y$ by $\tfrac{y}{\max_i y^\top a_i}$, we may assume without loss of generality that
$\max_i y^\top a_i=1$.

For every $x$ in the ellipsoid on the left-hand side in \eqref{eq:inner-directional-goal}, the Cauchy--Schwarz inequality
yields
\begin{equation*}
 y^\top x
 =\bigl(\widehat Q^{\tfrac12}y\bigr)^\top\bigl(\widehat Q^{-\tfrac12}x\bigr)
 \le\sqrt{y^\top \widehat Qy}\sqrt{x^\top \widehat Q^{-1}x}
 \le\sqrt{\alpha y^\top \widehat Qy}\,,
\end{equation*}
with equality attained above for
$x=\tfrac{\sqrt\alpha\,\widehat Qy}{\sqrt{y^\top \widehat Qy}}$.  Hence the left-hand side of \eqref{eq:inner-directional-goal} is
\begin{equation}\label{eq:inner-directional-maximum-left}
 \max_{x\in\sqrt\alpha\,{E}_{\widehat Q^{-1}}}y^\top x
 =\sqrt{\alpha y^\top \widehat Qy}\,.
\end{equation}
For the right-hand side of \eqref{eq:inner-directional-goal},  
\begin{equation}\label{eq:inner-directional-maximum-right}
 \max_{x\in S-c}y^\top x
 =\max_{s\in S}y^\top(s-c)
 =\max_i y^\top a_i-y^\top c
 =1-\frac{1-\alpha}{2}y^\top A\tilde u\,,
\end{equation}
where the last equality uses $\max_i y^\top a_i=1$ and $c=\tfrac{1-\alpha}{2}A\tilde u$.

For a given index $i$ it holds that $y^\top a_i\le1$.  On the other hand, since $-\alpha S\subseteq S$, one also has
$-\alpha a_i\in S$, so $-\alpha y^\top a_i\le\max_j y^\top a_j=1$ and hence
$y^\top a_i\ge-\tfrac1\alpha$.  Write $t_i:=y^\top a_i$.  Then
\begin{equation}\label{eq:pointwise-chord}
 \alpha t_i^2+(1-\alpha)t_i\le1\,,
\end{equation}
because
$1-(\alpha t_i^2+(1-\alpha)t_i)=(1-t_i)(1+\alpha t_i)\ge0$ for
$-\tfrac1\alpha\le t_i\le1$.  Multiplying \eqref{eq:pointwise-chord} by $\tilde u_i$ and summing over
$i=1,\ldots,m$ yields
\begin{equation*}
\alpha\sum_{i=1}^m\tilde u_i(y^\top a_i)^2
 +(1-\alpha)\sum_{i=1}^m\tilde u_i y^\top a_i
 \le\sum_{i=1}^m\tilde u_i=1\,,
\end{equation*}
where the last equality uses $e^\top\tilde u=1$.  In matrix notation, using $A\tilde u=\sum_i\tilde u_i a_i$ and
$A\widetilde U A^\top=\sum_i\tilde u_i a_i a_i^\top$, this is
\begin{equation}
 \alpha y^\top A\widetilde U A^\top y+(1-\alpha)y^\top A\tilde u \le1\,.
 \label{eq:finite-integrated-chord}
\end{equation} 
Therefore
\begin{align}
 \alpha y^\top \widehat Qy
 &=\alpha y^\top A\widetilde U A^\top y
   -\frac{\alpha(1-\alpha)}2(y^\top A\tilde u)^2
 \le1-(1-\alpha)y^\top A\tilde u
   -\frac{\alpha(1-\alpha)}2(y^\top A\tilde u)^2 \notag\\
 &=\left(1-\frac{1-\alpha}{2}y^\top A\tilde u\right)^2
   -\frac{1-\alpha^2}{4}(y^\top A\tilde u)^2
 \le\left(1-\frac{1-\alpha}{2}y^\top A\tilde u\right)^2\,.
 \label{eq:inner-square-bound}
\end{align}
Here the first equality expands $\widehat Q=A\widetilde U A^\top-\tfrac{1-\alpha}{2}(A\tilde u)(A\tilde u)^\top$, the first inequality is \eqref{eq:finite-integrated-chord}, and the last inequality uses $0<\alpha\le1$.  It follows from \eqref{eq:inner-directional-maximum-left} that the left-hand side of \eqref{eq:inner-square-bound} is the square of the left-hand side of \eqref{eq:inner-directional-goal}.  And it follows from \eqref{eq:inner-directional-maximum-right} that the right-hand side of \eqref{eq:inner-square-bound} is the square of the right-hand side of \eqref{eq:inner-directional-goal}.  The first maximum in \eqref{eq:inner-directional-goal} is nonnegative; also $A\tilde u\in S$ and the normalization $\max_i y^\top a_i=1$ imply that
$1-\tfrac{1-\alpha}{2}y^\top A\tilde u\ge\tfrac{1+\alpha}{2}>0$ and so the second maximum in \eqref{eq:inner-directional-goal} is nonnegative.  Therefore taking square roots in \eqref{eq:inner-square-bound} yields \eqref{eq:inner-directional-goal}, and hence \eqref{eq:dual-feasible-inner-containment}.
\end{proof}

\subsection{Proof of \texorpdfstring{\Cref{lm:rounding-w-representation}}{the rounding lemma for the W-representation}}

\begin{proof}[Proof of \Cref{lm:rounding-w-representation}]
Let $(Q^\star,c^\star)$ be an optimal solution of $(\mathrm P_{\lambda_\alpha})$.  By the change of variables
\eqref{eq:square-root-variables} it follows that $M:= (Q^\star)^{\tfrac12}$ and $z:=Mc^\star$ are optimal solutions of \eqref{eq:transformed-primal} whereby there exists $u$ for which the triplet $(M,z,u)$ satisfies the optimality conditions \eqref{eq:optimality-primal-feasibility}, \eqref{eq:optimality-dual-feasibility}, \eqref{eq:optimality-stationarity}, and \eqref{eq:optimality-complementary-slackness}, and in fact as shown earlier it holds that $e^\top u=n$.  Therefore $\tilde u:=\tfrac{u}{n}$ is feasible for the
reduced dual problem $(\mathrm D_{\lambda_\alpha}^{\rm red})$, and also it follows from \eqref{eq:optimality-stationarity}, \eqref{eq:reduced-dual-stationarity-2}, and the notation $\widehat Q:=Q_{\tilde u,\lambda_\alpha}$ from \Cref{lem:dual-inner} that
\begin{equation}\label{eq:optimal-primal-dual-matching}
 c^\star=\frac{1}{1+\lambda_\alpha}A\tilde u\,,
 \qquad
 (Q^\star)^{-1}=M^{-2} = n Q_{\tilde u,\lambda_\alpha} = n \widehat Q \,.
\end{equation}
When $\alpha=1$, \eqref{eq:optimal-primal-dual-matching} remains valid: indeed, $\lambda_\alpha=+\infty$ fixes $c^\star=0$, and the endpoint stationarity condition yields $(Q^\star)^{-1}=nQ_{\tilde u,\infty}=n\widehat Q$.

It therefore follows from \Cref{lem:dual-inner} that $ \sqrt{\alpha} E_{\widehat Q^{-1}} \subset S - c^\star$.  Now note from \eqref{eq:optimal-primal-dual-matching} that $\widehat Q^{-1} = n Q^\star$ whereby $\sqrt{\alpha} E_{\widehat Q^{-1}} = \sqrt{\frac{\alpha}{n}} E_{Q^\star}$ and therefore
$$ \sqrt{\frac{\alpha}{n}} E_{Q^\star} \subset S - c^\star\,,$$ which is the inscribed inclusion in \Cref{lm:rounding-w-representation}.

The outer containment statements of \Cref{lm:rounding-w-representation} follow by applying \Cref{lem:outer-containment} to
$(Q^\star,c^\star)$.  It also follows from \eqref{eq:alpha-lambda-coefficient} and \eqref{eq:optimal-primal-dual-matching} that
$c^\star\in\tfrac{1-\alpha}{2}S$.
\end{proof}

\section{Proof of \texorpdfstring{\Cref{thm:global-rounding}}{the main theorem}}\label{sec:polytope-to-convex-body}

\Cref{lm:rounding-w-representation} provides the guaranteed $\sqrt{\frac{n}{\mathop{\mathbf{sym}}\nolimits(S)}}$-rounding of \Cref{thm:global-rounding} in the case when $S$ is a polytope.  When $S$ is a general convex body we prove \Cref{thm:global-rounding} by approximating $S$ by a sequence of polytopes and taking limits.

\begin{proof}[Proof of \Cref{thm:global-rounding}]
For any $\varepsilon>0$ there exists a polytope $P$ that approximates $S$ in the sense that $S \subset P \subset S+\varepsilon B_2^n$.  Let $\{P_k\}$ be a sequence of polytopes converging to $S$ such that
$S\subseteq P_k\subseteq S +\varepsilon_k B_2^n$ for some sequence $\varepsilon_k\downarrow0$. After translating each $P_k$ so that a Minkowski center lies at the origin and applying \Cref{lm:rounding-w-representation} directly to $P_k$, it follows for each $k$ that there exist
 $Q_k\succ0$ and $c_k\in P_k$  such that
\begin{equation}\label{eq:convex-body-rounding-sandwich-2}
 \sqrt{\frac{\mathop{\mathbf{sym}}\nolimits(P_k)}{n}}\,{E}_{Q_k}
 \subseteq P_k-c_k
 \subseteq{E}_{Q_k}\,.
\end{equation}
The centers $c_k$ are uniformly bounded, and the ellipsoids $E_{Q_k}$ are uniformly bounded and nondegenerate. Therefore, there exists a subsequence, indexed by $k_j$, such that $c_{k_j}\to c$ and $Q_{k_j}\to Q$ for some $c$ and $Q$. 
Furthermore, since $P_{k_j}\to S$,  it holds that $\mathop{\mathbf{sym}}\nolimits(P_{k_j})\to\mathop{\mathbf{sym}}\nolimits(S)$. Taking limits in
\eqref{eq:convex-body-rounding-sandwich-2} along this subsequence  yields $
 \sqrt{\frac{\mathop{\mathbf{sym}}\nolimits(S)}{n}}\,{E}_Q
 \subseteq S-c
 \subseteq{E}_Q$.  
\end{proof}

\section{Center-regularized MVIE and rounding of a polytope in $H$-representation}\label{sec:h-representation}

In \Cref{lm:rounding-w-representation} we used the center-regularized MVCE problem \eqref{eq:center-regularized-loewner} to construct (and hence prove there exists) a $\sqrt{\frac{n}{\mathop{\mathbf{sym}}\nolimits(S)}}$-rounding of $S$ when $S$ is described in a $W$-representation (the weighting of points). This construction involved first translating the polytope so that the origin is a Minkowski center, and then solving \eqref{eq:center-regularized-loewner} (or, equivalently, \eqref{eq:transformed-primal}) to determine the desired ellipsoid.  And in Section \ref{compute paragraph} we described how to efficiently compute this rounding when $S$ is described in a $W$-representation.

However, in most cases in modeling with polytopes in operations research and related settings, the polytope is defined as the intersection of finitely many halfspaces, which we denote as an $H$-representation. In order to use the computational scheme based on \Cref{lm:rounding-w-representation} for a polytope in $H$-representation, one can in principle convert the polytope's representation from $H$ to $W$, but the conversion itself is generally a hard problem since it is typical for the $W$-representation to involve exponentially many points even when the number of halfspaces in the $H$-representation is relatively small. In order to compute a $\sqrt{\frac{n}{\mathop{\mathbf{sym}}\nolimits(S)}}$-rounding for a polytope in $H$-representation, in this section we present a center-regularized version of the MVIE optimization problem for polytopes in $H$-representation, together with a version of \Cref{lm:rounding-w-representation} that is applicable directly to polytopes in $H$-representation. The construction is analogous to the center-regularized MVCE problem \eqref{eq:center-regularized-loewner}, but it is naturally better suited to computation for polytopes in $H$-representation.

Similar to the development in Section \ref{sec:v-representation}, we suppose that $S$ contains the origin in its interior, and is described by the intersection of $\ell$ halfspaces, i.e., it is given in $H$-representation as
\begin{equation}\label{eq:john-h-representation}
 S=\left\{x\in\mathbb{R}^n\,\middle|\,h_i^\top x\le1,\quad i=1,\ldots,\ell\right\}\,,
\end{equation}
where $h_1,\ldots,h_\ell$ are the listed nonzero halfspace normals. Here it is apparent that the origin is an interior point of $S$.  Also, since $S$ is defined to be a polytope, it is therefore bounded, and hence there is no ray of $S$, which means that there is no nonzero $d$ for which $h_i^\top d \le 0$ for all $i = 1, \ldots, \ell$.

We first recall the standard MVIE optimization problem.
For
$Q\succ0$, and because $\max_{x\in {E}_Q}h_i^\top x=\sqrt{h_i^\top Q^{-1}h_i}$, we have:
\begin{equation}\label{eq:john-inscribed-condition}
\begin{aligned}
 c+{E}_Q\subseteq S
 &\quad\Longleftrightarrow\quad
 h_i^\top c+\sqrt{h_i^\top Q^{-1}h_i}\le1
 &&\text{for }i=1,\ldots,\ell\,,\\
 &\quad\Longleftrightarrow\quad
 h_i^\top Q^{-1}h_i\le(1-h_i^\top c)^2
 \quad\text{and}\quad h_i^\top c\le1
 &&\text{for }i=1,\ldots,\ell\,.
\end{aligned}
\end{equation} 
Using the change of variables $X = Q^{-\tfrac12}$, the standard MVIE is the solution of the following optimization problem:
\begin{equation}\label{eq:free-john-primal}
\begin{aligned}
 \operatorname*{minimize}_{X,c}\quad&-2\log\det X\\
 \text{subject to}\quad
 &\lVert Xh_i\rVert_2^2\le(1-h_i^\top c)^2\,,
 \qquad h_i^\top c\le1\,,
 \quad i=1,\ldots,\ell\,,\\
 &X\succ0,\qquad c\in\mathbb{R}^n\,,
\end{aligned}
\end{equation}
and the corresponding maximum-volume inscribed ellipsoid is $c+{E}_{X^{-2}}$.  This convex formulation and a polynomial-time algorithm for \eqref{eq:free-john-primal} appear in \citet{KhachiyanTodd1993}; related algorithms for solving \eqref{eq:free-john-primal} include
\citet{Anstreicher2002,ZhangGao2003}.

To regularize the center, we rewrite the first constraint in \eqref{eq:free-john-primal} as
$\lVert Xh_i\rVert_2^2-(h_i^\top c)^2+2h_i^\top c\le1$ and add the center penalty
$\lambda(h_i^\top c)^2$ to the $i$th constraint, which yields the following problem.
\begin{equation}\tag{$\mathrm J_\lambda$}\label{eq:john-primal}
\stepcounter{equation}
\begin{aligned}
 \operatorname*{minimize}_{X,c}\quad&-2\log\det X\\
 \text{subject to}\quad
 &\lVert Xh_i\rVert_2^2+(\lambda-1)(h_i^\top c)^2+2h_i^\top c\le1\,,
 \quad i=1,\ldots,\ell\,,\\
 &X\succ0,\qquad c\in\mathbb{R}^n\,.
\end{aligned}
\end{equation}
The objective maximizes the volume of ${E}_{X^{-2}}$.  We only consider the penalizing parameter in the range $\lambda\ge1$, whereby feasibility implies 
$h_i^\top c<1$ for every $i$, and therefore 
$c\in\operatorname{int} S$ for all feasible $X,c$.  Furthermore the objective and all
constraints in \eqref{eq:john-primal} are convex functions of $X,c$.  At $\lambda=+\infty$, we fix $c=0$ in which case the constraints of \eqref{eq:john-primal} become $\lVert Xh_i\rVert_2^2\le1$ for every $i$.  Analogous to \Cref{lm:rounding-w-representation} we have:

\begin{lemma}\label{lm:john-rounding}
Let $S$ be a polytope defined as in \cref{eq:john-h-representation}, and let $\alpha$ and $\lambda_\alpha$ be defined as in
\cref{eq:alpha-lambda}.  The center-regularized problem \eqref{eq:john-primal} with
$\lambda=\lambda_\alpha$ attains its optimal value.  For an optimal solution $(X^\star,c^\star)$ define $Q^\star:=(X^\star)^{-2}$. Then $(Q^\star, c^\star)$ satisfies
\begin{equation*}
 {E}_{Q^\star}
 \subseteq S-c^\star
 \subseteq\sqrt{\frac n\alpha}\,{E}_{Q^\star}\,.
\end{equation*}
Moreover, the ellipsoid in the inner containment can be enlarged to $
 \big(
  \min_{i}
  \frac{1-h_i^\top c^\star}{\lVert X^\star h_i\rVert_2}
 \big){E}_{Q^\star}$.
\end{lemma}

If $S$ is a polytope with nonempty interior defined in $H$-representation, then translating a Minkowski center of $S$ to the origin and then applying \Cref{lm:john-rounding} yields a $\sqrt{\tfrac{n}{\mathop{\mathbf{sym}}\nolimits(S)}}$-rounding of $S$ and provides an alternate proof of \Cref{thm:global-rounding} when $S$ is a polytope.

As with the center-regularized MVCE, if $0$ is a Minkowski center of $S$ and $\lambda$ is chosen as in \eqref{eq:alpha-lambda}, then the center-regularized MVIE problem \eqref{eq:john-primal} recovers the MVIE at both extrema of $\mathop{\mathbf{sym}}\nolimits(S)$, namely $\tfrac1n$ and $1$. \Cref{lm:john-rounding} also recovers the standard MVIE rounding guarantees at both extrema of $\mathop{\mathbf{sym}}\nolimits(S)$ as well.
 
Computing the center-regularized MVIE to produce the desired rounding involves first computing the Minkowski center and then solving \eqref{eq:john-primal}.
A Minkowski center $x^s$ and the value $\mathop{\mathbf{sym}}\nolimits(S)$ can be computed by solving $\ell+1$ linear optimization problems each of which involves $\ell$ inequality constraints in at most $n+1$ unrestricted variables, see \citep[][]{BelloniFreund2008}. After translating $x^s$ to the origin, one then solves \eqref{eq:john-primal} for an optimal solution pair $X^\star, c^\star$ and then converts $Q^\star = (X^\star)^{-2}$ and $c^* \leftarrow c^* + x^s$.  Notice that \eqref{eq:john-primal} is a convex problem with a logarithmic determinant objective, $\ell$ second-order cone constraints, and one positive semidefinite cone constraint, and so is amenable to solution by standard interior-point methods \citep{VandenbergheBoydWu1998}.

\ 

The remainder of this section is devoted to properties of \eqref{eq:john-primal} and then a proof of \Cref{lm:john-rounding}.  

\subsection{Lagrange dual and optimality conditions of \texorpdfstring{\Cref{eq:john-primal}}{the center-regularized MVIE problem}}

For notational convenience, collect the halfspace normals in the column-stacked matrix
\begin{equation*}
 H:=[h_1,\ldots,h_\ell]\in\mathbb{R}^{n\times\ell}\,.
\end{equation*}

For $1<\lambda<+\infty$, assign a multiplier $v_i\ge0$ to the $i$th constraint of
\eqref{eq:john-primal}.  The Lagrangian is
\begin{equation*}
 L(X,c,v)=-2\log\det X+\operatorname{tr}\!\left(XH\operatorname{Diag}(v)H^\top X\right)
 +(\lambda-1)c^\top H\operatorname{Diag}(v)H^\top c+2c^\top Hv-e^\top v\,.
\end{equation*}
The primal problem can be written as $\min_{X\succ0,\,c\in\mathbb{R}^n} \, \max_{v\ge0}L(X,c,v)$.
The Lagrange dual interchanges the min and max operations.  The objective function of the dual is then the inner minimization problem 
\begin{equation*}
 g(v):=\min_{X\succ0,\,c\in\mathbb{R}^n}L(X,c,v)
\end{equation*}
(which will itself be optimized over $v\ge0$).  We can compute $g(v)$ explicitly by minimizing first over $X$ and then over $c$.
If $H\operatorname{Diag}(v)H^\top$ is singular, the minimization in the definition of $g(v)$ above is unbounded below along a null direction.
Thus only positive definite matrices can occur in the domain of $g$.
When $H\operatorname{Diag}(v)H^\top\succ0$, the terms of $L(X,c,v)$ involving $X$ can be isolated in the minimization over $X\succ 0$ as:
\begin{equation*}
 \min_{X\succ0}\left\{-2\log\det X
 +\operatorname{tr}\!\left(XH\operatorname{Diag}(v)H^\top X\right)\right\}
 =n+\log\det\!\left(H\operatorname{Diag}(v)H^\top\right),
\end{equation*}
and its unique minimizer is
$X=(H\operatorname{Diag}(v)H^\top)^{-\tfrac12}$.  
Next, minimization of $L(X,c,v)$ over $c$ yields
$c=-(\lambda-1)^{-1}(H\operatorname{Diag}(v)H^\top)^{-1}Hv$, and substituting back in the Lagrangian yields the dual problem:
\begin{equation}\tag{$\mathrm D_\lambda^{\mathrm J}$}\label{eq:john-dual}
\stepcounter{equation}
\begin{aligned}
 \operatorname*{maximize}_{v\in\mathbb{R}_+^\ell}\quad
 &n+\log\det\!\left(H\operatorname{Diag}(v)H^\top\right)-e^\top v
 -\frac1{\lambda-1}(Hv)^\top
 \left(H\operatorname{Diag}(v)H^\top\right)^{-1}Hv\\
 \text{subject to}\quad&H\operatorname{Diag}(v)H^\top\succ0\,.
\end{aligned}
\end{equation}
For the rest of this section let us denote:
\begin{equation}\label{eq:john-Gv}
 G_v:=H\operatorname{Diag}(v)H^\top\,, 
\end{equation}
and notice that by construction the objective function in \eqref{eq:john-dual} is concave in $v$ on the domain $G_v\succ0$.

The primal problem \eqref{eq:john-primal} has compact sublevel sets since $S$ is a polytope (and hence is bounded and so there is no nonzero value $\tilde c$ for which $H^\top \tilde c \le 0$) and therefore \eqref{eq:john-primal} attains its optimal value.  The point
$(X,c)=(\varepsilon I,0)$ is strictly feasible for sufficiently small $\varepsilon>0$, so the Slater condition
holds for \eqref{eq:john-primal} which implies both strong duality and dual attainment.  The Karush--Kuhn--Tucker optimality conditions are therefore necessary
and sufficient \citep[see also][for the dual problem of the ordinary MVIE]{Guertuna2010}.
Thus $(X,c)$ and $v$ are optimal for \eqref{eq:john-primal} and \eqref{eq:john-dual} if and only if they satisfy:
\begingroup
\begin{flalign}
 &\text{Primal feasibility:}\quad  
 X\succ0,\quad c\in\mathbb{R}^n,\quad
 \lVert Xh_i\rVert_2^2+(\lambda-1)(h_i^\top c)^2+2h_i^\top c\le1
 \quad(i=1,\ldots,\ell)\,. &&\notag\\
 &\text{Dual feasibility:}\quad  
 v\in\mathbb{R}_+^\ell,\quad G_v\succ0\,. &&\notag\\
 &\text{Stationarity:}\quad  
 X=G_v^{-\tfrac{1}{2}},\qquad Hv=-(\lambda-1)G_vc\,. &&\label{eq:john-optimality-stationarity}\\
 &\text{Complementary slackness:}\quad  
 v_i\bigl(\lVert Xh_i\rVert_2^2+(\lambda-1)(h_i^\top c)^2
           +2h_i^\top c-1\bigr)=0\quad(i=1,\ldots,\ell)\,. &&\notag
\end{flalign}
\endgroup
At $\lambda=+\infty$, one replaces the primal and complementary slackness expressions by their fixed center
forms $\lVert Xh_i\rVert_2^2\le1$, fixes $c=0$, and deletes the second stationarity equation.

From the optimality conditions, we also obtain the following equality for every optimal solution $v^\star$ of \eqref{eq:john-dual}.

\begin{lemma}\label{lem:john-dual-scale}
For $1<\lambda<+\infty$, every optimizer $v^\star$ of \eqref{eq:john-dual} satisfies
\begin{equation}\label{eq:john-dual-scale}
 e^\top v^\star
 +\frac1{\lambda-1}(Hv^\star)^\top G_{v^\star}^{-1}Hv^\star=n\,.
\end{equation}
At $\lambda=+\infty$, every dual optimizer $v^\star$ satisfies $e^\top v^\star=n$.
\end{lemma}

\begin{proof}
If $v^\star$ is dual optimal, then $tv^\star$ is dual feasible for every $t>0$.  Along this ray, the dual objective is
\begin{equation}\label{eq:john-dual-objective-ray}
 n+\log\det G_{v^\star}+n\log t
 -t\left(e^\top v^\star
 +\frac1{\lambda-1}(Hv^\star)^\top G_{v^\star}^{-1}Hv^\star\right).
\end{equation}
The derivative of \eqref{eq:john-dual-objective-ray} at the maximizer $t=1$ is zero, which yields
\eqref{eq:john-dual-scale}.  At $\lambda=+\infty$, the final term in parentheses is absent and the same
argument yields $e^\top v^\star=n$.
\end{proof}

\subsection{Inner and outer ellipsoids from primal and dual feasibility}

The following inner containment result is based solely on primal feasibility.

\begin{lemma}\label{lem:john-inner}
Let $1<\lambda\le+\infty$.  Every feasible $(X,c)$ of \eqref{eq:john-primal} satisfies
\begin{equation*}
 {E}_{X^{-2}}
 \subseteq
 \left(
  \min_{i}
  \frac{1-h_i^\top c}{\lVert Xh_i\rVert_2}
 \right){E}_{X^{-2}}
 \subseteq S-c\,.
\end{equation*}
\end{lemma}

\begin{proof}
For finite $\lambda$, rearranging the quadratic constraints in \eqref{eq:john-primal} yields
$\lVert Xh_i\rVert_2^2+\lambda(h_i^\top c)^2\le(1-h_i^\top c)^2$, and hence
$\lVert Xh_i\rVert_2^2\le(1-h_i^\top c)^2$.  Feasibility and $\lambda >1$ also imply
$1-h_i^\top c>0$. Thus
$\tfrac{1-h_i^\top c}{\lVert Xh_i\rVert_2}\ge1$.  At $\lambda=+\infty$, the center is zero and the constraints
$\lVert Xh_i\rVert_2^2\le1$ yield the same inequality.  The factor $\min_{i}
  \frac{1-h_i^\top c}{\lVert Xh_i\rVert_2}$ is therefore at least $1$, which
proves the first containment.  Moreover, for every $h_i$ we have
\begin{equation*}
 h_i^\top c
 +\left(
   \min_{j}
   \frac{1-h_j^\top c}{\lVert Xh_j\rVert_2}
  \right)\lVert Xh_i\rVert_2
 \le1 \,,
\end{equation*}
whereby applying \eqref{eq:john-inscribed-condition} with $Q=\left(\min_j\tfrac{1-h_j^\top c}{\lVert Xh_j\rVert_2}\right)^{-2}X^{-2}$ proves the second containment.
\end{proof}
 
\begin{lemma}\label{lem:john-outer}
Let $0<\alpha<1$ and $\lambda=\lambda_\alpha$ be defined as in \cref{eq:alpha-lambda}.  Suppose that $v$ is feasible for
\eqref{eq:john-dual} and also satisfies
\begin{equation}\label{eq:john-dual-bound}
 e^\top v+\frac1{\lambda-1}(Hv)^\top G_v^{-1}Hv\le n\,.
\end{equation}
Set $c_v:=-(\lambda-1)^{-1}G_v^{-1}Hv$.  Then
\begin{equation}\label{eq:john-outer-containment}
 S-c_v\subseteq\sqrt{\frac n\alpha}\,{E}_{G_v}\,.
\end{equation}
For $\alpha=1$, every $v\in\mathbb{R}_+^\ell$ satisfying $G_v\succ0$ and $e^\top v\le n$ satisfies
$S\subseteq\sqrt n\,{E}_{G_v}$.
\end{lemma}

\begin{proof}
Suppose first that $0<\alpha<1$.  By the definition of ${E}_{G_v}$, \eqref{eq:john-outer-containment} is equivalent to
proving that every $x\in S$ satisfies
\begin{equation}\label{eq:john-outer-goal}
 (x-c_v)^\top G_v(x-c_v)\le\frac n\alpha\,.
\end{equation}
For each $i=1, \ldots, \ell$, if $x\in S$ it holds that $h_i^\top x\le1$ and also $-\alpha x\in S$ whereby $h_i^\top x\ge-\tfrac1\alpha$.  Consequently,
$(1-h_i^\top x)(1+\alpha h_i^\top x)\ge0$, and expanding yields
\begin{equation}\label{eq:john-pointwise-chord}
 \alpha(h_i^\top x)^2+(1-\alpha)h_i^\top x\le1\,.
\end{equation}
Multiplying these inequalities by the nonnegative weights $v_i$ and summing yields
\begin{equation}\label{eq:john-weighted-chord}
 \sum_{i=1}^\ell v_i\bigl(\alpha(h_i^\top x)^2+(1-\alpha)h_i^\top x\bigr)
 =\alpha x^\top G_vx+(1-\alpha)(Hv)^\top x\le e^\top v\,.
\end{equation}
Here the equality follows from \eqref{eq:john-Gv} and the inequality follows from
\eqref{eq:john-pointwise-chord} and $v\ge0$.

To recover the quadratic form in \eqref{eq:john-outer-goal}, use the definition of $c_v$ and
$ (1-\alpha)(\lambda_\alpha-1)=2\alpha$ (by the definition of $\lambda_\alpha$) to obtain:
$(1-\alpha)Hv=-2\alpha G_vc_v$.  Substituting this identity into \eqref{eq:john-weighted-chord} and adding
$\alpha c_v^\top G_vc_v$ to both sides yields
\begin{equation}\label{eq:john-outer-quadratic-bound}
 \alpha(x-c_v)^\top G_v(x-c_v)
 \le e^\top v+\alpha c_v^\top G_vc_v\,.
\end{equation}
Since $\lambda=\lambda_\alpha$, we have
$\lambda_\alpha-1=\tfrac{2\alpha}{1-\alpha}\ge\alpha$.  Therefore
\begin{equation*}
\begin{aligned}
 e^\top v+\alpha c_v^\top G_vc_v
 \le e^\top v+(\lambda_\alpha-1)c_v^\top G_vc_v  
 =e^\top v+\frac1{\lambda-1}(Hv)^\top G_v^{-1}Hv  
 \le n\,.
\end{aligned}
\end{equation*}
Here the equality uses $Hv=-(\lambda-1)G_vc_v$, which follows from the definition of $c_v$, and the final
inequality is \eqref{eq:john-dual-bound}.  Combining this bound with
\eqref{eq:john-outer-quadratic-bound} and dividing by $\alpha$ proves \eqref{eq:john-outer-goal}.  As $x\in S$
was chosen arbitrarily this shows the desired containment.

For $\alpha=1$ and the hypothesis that $e^\top v\le n$, the relation $-S = S$ implies that
$|h_i^\top x|\le1$ for every $x\in S$ and every $i=1, \ldots, \ell$.  Then
$x^\top G_vx=\sum_{i=1}^\ell v_i(h_i^\top x)^2\le e^\top v\le n$, which then implies that $S\subseteq\sqrt n\,{E}_{G_v}$.
\end{proof}

\subsection{Proof of \texorpdfstring{\Cref{lm:john-rounding}}{the rounding lemma for the H-representation}}

\begin{proof}[Proof of \Cref{lm:john-rounding}]
Let $(X^\star,c^\star)$ be an arbitrary optimal solution of $(\mathrm J_{\lambda_\alpha})$, and set
$Q^\star:=(X^\star)^{-2}$.  Strong duality and dual attainment guarantee that there exists a dual optimizer $v^\star$
for which $(X^\star,c^\star, v^\star)$ satisfy the optimality conditions \eqref{eq:john-optimality-stationarity}.  Utilizing \Cref{lem:john-dual-scale}, it also holds that $v^\star$
satisfies the additional condition on $v^\star$ in \Cref{lem:john-outer} (namely \eqref{eq:john-dual-bound} in the case when $\alpha <1$; and it is $e^\top v^\star \le n$ in the case when $\alpha =1$).

Suppose first that $0<\alpha<1$.  The stationarity conditions in
\eqref{eq:john-optimality-stationarity} yield
\begin{equation*}
 c^\star=-\frac1{\lambda_\alpha-1}G_{v^\star}^{-1}Hv^\star=c_{v^\star}\,,
 \qquad
 X^\star=G_{v^\star}^{-\tfrac12}\,.
\end{equation*}
In particular, $G_{v^\star}=(X^\star)^{-2}=Q^\star$.  The two inner containment assertions then follow from
\Cref{lem:john-inner}, the outer containment follows from \Cref{lem:john-outer}, and the two stationarity identities show that the inner and outer ellipsoids have the same center and shape.

At $\alpha=1$, the center is fixed at zero.  The stationarity condition yields
$G_{v^\star}=(X^\star)^{-2}=Q^\star$, and \Cref{lem:john-inner,lem:john-outer} imply the desired inner and outer containment with factor $\sqrt n$.
\end{proof}

\section{Sharpness of the symmetry-dependence in rounding}\label{sec:sharpness-family}

In this section, we show that the rounding factor $\sqrt{\tfrac{n}{\mathop{\mathbf{sym}}\nolimits(S)}}$ in \Cref{thm:global-rounding} is nearly tight in its dependence on dimension and symmetry.  In the case when $\tfrac{n+1}{1+\mathop{\mathbf{sym}}\nolimits(S)}$ is an integer, we show by explicit construction that the factor $\sqrt{\tfrac{n}{\mathop{\mathbf{sym}}\nolimits(S)}}$ is tight.  More generally, for every dimension $n$ and every admissible symmetry value, the factor is sharp up to a universal constant $\sqrt{\tfrac{2}{3}} \approx .816$.

Fix an integer $n\ge1$, and let $\Delta_n:=\mathop{\mathbf{conv}}\nolimits\{a_1,\ldots,a_{n+1}\}$ where
$a_1,\ldots,a_{n+1}\in\mathbb{R}^n$ satisfy
\begin{equation}\label{eq:regular-simplex-normalization}
    \|a_i\|_2=1
    \quad\text{for all }i,
    \qquad
    \langle a_i,a_j\rangle=-\frac1n
    \quad\text{for all }i\neq j\,.
\end{equation}
It is easy to show that with \eqref{eq:regular-simplex-normalization}, $\sum_{i=1}^{n+1}a_i=0$ and all edges have the same length (\Cref{lemma:simplex-frame-identities}).
We therefore call $\Delta_n$ the \emph{normalized regular simplex}.
For $n=1$, $a_1,\ldots,a_{n+1}$ could be the two
points $a_1=1$ and $a_2=-1$.  For $n=2$, they could be the vertices $(1,0)$ and
$(-\tfrac12,\pm\tfrac{\sqrt3}{2})$ of an equilateral triangle on the unit circle.

The family of convex bodies we study is the convex hull of $\Delta_n$ and its scaled reflection:
\begin{equation}\label{eq:def-S-n-alpha}
 S_{n,\alpha}:=\mathop{\mathbf{conv}}\nolimits\bigl(\Delta_n\cup(-\alpha\Delta_n)\bigr)
 =\mathop{\mathbf{conv}}\nolimits\{a_i,-\alpha a_i:i=1,\ldots,n+1\}\subset\mathbb{R}^n\,.
\end{equation}
 
\begin{theorem}\label{thm:sharpness-family}
For every $n\ge1$ and every $\alpha\in[\tfrac1n,1]$, the body $S_{n,\alpha}$ satisfies $\mathop{\mathbf{sym}}\nolimits(0,S_{n,\alpha})=\mathop{\mathbf{sym}}\nolimits(S_{n,\alpha})=\alpha$.
Moreover, every $c\in\mathbb{R}^n$, origin-centered ellipsoid $E$, and $\rho\ge1$ satisfying
$E\subseteq S_{n,\alpha}-c\subseteq\rho E$ obey
\begin{equation}\label{eq:uniform-family-rounding-lower-bound}
 \rho\ge\sqrt{\frac{2n}{3\alpha}}\,.
\end{equation}
If, in addition, $\tfrac{n+1}{1+\alpha}$ is an integer, then 
\begin{equation}\label{eq:family-rounding-lower-bound}
 \rho\ge\sqrt{\frac n\alpha}\,.
\end{equation}
\end{theorem}

This theorem implies that the rounding factor upper bound $\sqrt{\tfrac{n}{\mathop{\mathbf{sym}}\nolimits(S)}}$ established in \Cref{thm:global-rounding} is nearly tight in its dependency on $n$ and $\mathop{\mathbf{sym}}\nolimits(S)$ for all possible $n$ and $\mathop{\mathbf{sym}}\nolimits(S)$, and exactly tight when $\tfrac{n+1}{1+\mathop{\mathbf{sym}}\nolimits(S)}$ is an integer.

\citet{BrandenbergGonzalezMerinoGrundbacher2026} concurrently and independently show that their Banach--Mazur
distance bound is tight when $\frac{1}{\mathop{\mathbf{sym}}\nolimits(S)}$ is an integer and also a divisor of $n$.
Since the Banach--Mazur distance to the Euclidean ball is a lower
bound on the rounding factor, their
examples also yield the lower bound
$\sqrt{\tfrac{n}{\mathop{\mathbf{sym}}\nolimits(S)}}$ when $\frac{1}{\mathop{\mathbf{sym}}\nolimits(S)}$ is an integer divisor of $n$.
 
In the rest of this section we prove \Cref{thm:sharpness-family}.
We begin with two facts about the normalized regular simplex.

\begin{lemma}\label{lemma:simplex-frame-identities}
For the normalized regular simplex
$\Delta_n=\mathop{\mathbf{conv}}\nolimits\{a_1,\ldots,a_{n+1}\}\subset\mathbb{R}^n$ defined above,
one has $\sum_{i=1}^{n+1}a_i=0$, all edges of $\Delta_n$ have the same
length $\sqrt{\tfrac{2(n+1)}{n}}$, and
$\sum_{i=1}^{n+1}a_i a_i^\top=\frac{n+1}{n}I_n$.
\end{lemma}

\begin{proof}
Let $s:=\sum_i a_i$.  For every $j$,
$\langle s,a_j\rangle=1+\sum_{i\ne j}\langle a_i,a_j\rangle
=1-n(\tfrac{1}{n})=0$.  Since the vertices span $\mathbb{R}^n$, $s=0$.
Moreover, for $i\ne j$,
$\|a_i-a_j\|_2^2=2-2\langle a_i,a_j\rangle
=\tfrac{2(n+1)}{n}$, so all edges have the same length $\sqrt{\tfrac{2(n+1)}{n}}$.

Let $H:=\sum_i a_i a_i^\top$.  Fix $j$. In $Ha_j$, the $i$-th term $a_i a_i^\top a_j$ equals $a_j$ when $i=j$, while $a_i a_i^\top a_j$ equals $-\tfrac{a_i}{n}$ when $i\ne j$.
Since $\sum_i a_i=0$, one has $\sum_{i\ne j}a_i=-a_j$.  Therefore
$Ha_j=a_j-\frac1n\sum_{i\ne j}a_i=\frac{n+1}{n}a_j$.
The vertices span $\mathbb{R}^n$, so $H=\frac{n+1}{n}I_n$.
\end{proof} 

Let $\Pi$ denote the group of all permutations of $\{1,\ldots,n+1\}$.

\begin{lemma} \label{lem:simplex-permutation-symmetries}
For every permutation $\pi\in \Pi$, there exists a unique linear map $T_\pi:\mathbb{R}^n\to\mathbb{R}^n$ such that
$T_\pi a_i=a_{\pi(i)}$ for every $i$.  Moreover, $T_\pi$ is orthogonal,
$T_\pi S_{n,\alpha}=S_{n,\alpha}$, and $\sum_{\pi\in \Pi}T_\pi=0$.
\end{lemma}

\begin{proof}
First, $a_1,\ldots,a_n$ form a basis of $\mathbb{R}^n$. 
Fix $\pi\in \Pi$.  There is therefore a unique linear map $T_\pi$ satisfying
$T_\pi a_i=a_{\pi(i)}$ for $i=1,\ldots,n$.  Since $\sum_{i=1}^{n+1}a_i=0$, linearity gives
$T_\pi a_{n+1} = - \sum_{i=1}^n T_\pi a_i
 = -\sum_{i=1}^n a_{\pi(i)}
 = a_{\pi(n+1)},$
so $T_\pi a_i=a_{\pi(i)}$ for every $i$.
Since $T_\pi$ permutes both sets of vertices $\{a_i\}$ and $\{-\alpha a_i\}$, taking convex hulls
gives $T_\pi S_{n,\alpha}=S_{n,\alpha}$.

The permutation preserves every inner product between vertices.  More explicitly, if
$x=\sum_{i=1}^n c_i a_i$ and $y=\sum_{j=1}^n d_j a_j$, then
\begin{align*}
 \langle T_\pi x,T_\pi y\rangle
 &=\sum_{i,j=1}^n c_i d_j\langle a_{\pi(i)},a_{\pi(j)}\rangle
 =\sum_{i,j=1}^n c_i d_j\langle a_i,a_j\rangle
 =\langle x,y\rangle\,.
\end{align*}
This proves that $T_\pi$ is orthogonal.

Finally, for each $i$, every vertex occurs exactly $n!$ times among the vectors $a_{\pi(i)}$.  Hence
$(\sum_{\pi\in \Pi}T_\pi)a_i=n!\sum_{j=1}^{n+1}a_j=0$.
The vectors $a_1,\ldots,a_{n+1}$ span $\mathbb{R}^n$, and hence the linear map
$\sum_{\pi\in \Pi}T_\pi$ is zero.
\end{proof}

\subsection{Symmetry of \texorpdfstring{$S_{n,\alpha}$}{S(n, alpha)}}

In this subsection, we show that the symmetry of $S_{n,\alpha}$ is $\alpha$.

\begin{lemma}\label{lem:family-symmetry}
The origin is a Minkowski center of $S_{n,\alpha}$, and
\begin{equation*}
 \mathop{\mathbf{sym}}\nolimits(0,S_{n,\alpha})=\mathop{\mathbf{sym}}\nolimits(S_{n,\alpha})=\alpha\,.
\end{equation*}
\end{lemma}

\begin{proof} 
Scale $S_{n,\alpha}$ by $-\alpha$. By definition in \eqref{eq:def-S-n-alpha}, the resulting set is the convex hull of $-\alpha\Delta_n$ and $\alpha^2\Delta_n$. Note that $0\in \Delta_n$ and $ 0 < \alpha \le 1$ so $\alpha^2\Delta_n\subseteq\Delta_n$. Therefore
$-\alpha S_{n,\alpha}\subseteq S_{n,\alpha}$ by \eqref{eq:def-S-n-alpha}.
Thus $\mathop{\mathbf{sym}}\nolimits(0,S_{n,\alpha})\ge\alpha$.

Next, we prove $\mathop{\mathbf{sym}}\nolimits(S_{n,\alpha}) \le \alpha$.
Fix any vertex $a_i$. By \eqref{eq:regular-simplex-normalization},
\begin{equation}\label{eq:lem-family-symmetry-1}
 \min_{x\in S_{n,\alpha}}\langle a_i,x\rangle
 = \min\left\{\min_j\langle a_i,a_j\rangle,\min_j\langle a_i,-\alpha a_j\rangle\right\} = \min\left\{-\frac{1}{n},-\alpha\right\}=
 -\alpha\,.
\end{equation}
Here the last equality uses $\alpha\ge\tfrac1n$.
Suppose that $c\in S_{n,\alpha}$ and $\beta>0$ satisfy $\mathop{\mathbf{sym}}\nolimits(c,S_{n,\alpha})=\beta$. Then
\begin{equation*}
 (1+\beta)c-\beta S_{n,\alpha}\subseteq S_{n,\alpha}\,.
\end{equation*}
Since $a_i\in S_{n,\alpha}$, the point $(1+\beta)c-\beta a_i$ belongs to $S_{n,\alpha}$.  Substituting $x=(1+\beta)c-\beta a_i$ back into \eqref{eq:lem-family-symmetry-1} gives
 $(1+\beta)\langle a_i,c\rangle-\beta\ge-\alpha$ for each $i=1,\ldots,n+1$.
Summing these inequalities for all $i=1,\ldots,n+1$ and using $\sum_i a_i=0$ gives $-(n+1)\beta\ge-(n+1)\alpha$, and hence
$\beta\le\alpha$.  Thus no center has symmetry greater than $\alpha$.  Together with
$\mathop{\mathbf{sym}}\nolimits(0,S_{n,\alpha})\ge\alpha$, this finishes the proof.
\end{proof}

It should be noted that a consequence of the general asymmetry comparison result in \citet[Theorem~6.1]{BrandenbergKoenig2015} is the following:
if $K$ is Minkowski centered at the origin and
$\tau\in[\mathop{\mathbf{sym}}\nolimits(K),1]$, then the origin is a Minkowski center of
$\mathop{\mathbf{conv}}\nolimits(K\cup(-\tau K))$ and $\mathop{\mathbf{sym}}\nolimits\!\left(\mathop{\mathbf{conv}}\nolimits(K\cup(-\tau K))\right)=\tau$.
Taking $K=\Delta_n$ and $\tau=\alpha$ gives
\Cref{lem:family-symmetry}.  For completeness, we still included the above self-contained proof. 

\subsection{MVCE and MVIE of \texorpdfstring{$S_{n,\alpha}$}{S(n, alpha)}}
 
In this subsection, we show that the MVCE and MVIE of $S_{n,\alpha}$ are both Euclidean balls centered at the origin. By computing their radii, we can obtain the best possible rounding factor of $S_{n,\alpha}$.

\begin{lemma}\label{lem:exact-extremal-ellipsoids}
For every $n\ge1$ and $\alpha\in[\tfrac1n,1]$, let ${E}_L$ denote the MVCE of
$S_{n,\alpha}$.  Then ${E}_L=B_2^n$.
Let ${E}_J$ denote the MVIE of $S_{n,\alpha}$, and
write ${E}_J=r_JB_2^n$.  Its radius satisfies
$\sqrt{\tfrac\alpha n}\le r_J\le\sqrt{\tfrac{3\alpha}{2n}}$.  If, in addition,
$\tfrac{n+1}{1+\alpha}$ is an integer, then $r_J=\sqrt{\tfrac\alpha n}$.
\end{lemma}
 
Before proving this lemma, we first show that both the MVCE and MVIE of $S_{n,\alpha}$ are Euclidean balls centered at the origin.

\begin{lemma}\label{lem:spherical-extremal-ellipsoids}
For every $n\ge1$ and $\alpha\in[\tfrac1n,1]$, let ${E}_J$ and ${E}_L$ denote the MVIE and
MVCE of $S_{n,\alpha}$, respectively.  There exist radii $r_J,r_L>0$ such that
${E}_J=r_JB_2^n$ and ${E}_L=r_LB_2^n$.
\end{lemma}

\begin{proof}

The MVIE and MVCE are unique \citep{Todd2016}. Let $E$ denote either $E_J$ or $E_L$. By \Cref{lem:simplex-permutation-symmetries}, $T_\pi S_{n,\alpha}=S_{n,\alpha}$ and $T_\pi$ is orthogonal for every $\pi\in\Pi$. If $E=E_J$, then $E\subseteq S_{n,\alpha}$ implies $T_\pi E\subseteq T_\pi S_{n,\alpha}=S_{n,\alpha}$. Moreover, since $T_\pi$ is orthogonal, $T_\pi E$ has the same volume as $E$. Hence $T_\pi E$ is also an MVIE of $S_{n,\alpha}$, and uniqueness gives $T_\pi E=E$. Similarly, if $E=E_L$, then $S_{n,\alpha}\subseteq E$ implies $S_{n,\alpha}=T_\pi S_{n,\alpha}\subseteq T_\pi E$, and the same volume and uniqueness argument shows that $T_\pi E=E$. Therefore,
\begin{equation}\label{eq:extremal-ellipsoid-invariance}
T_\pi E=E \qquad \text{for all }\pi\in\Pi.
\end{equation}

Let $c$ be the center of $E$. The center of $T_\pi E$ is $T_\pi c$. Since an ellipsoid has a unique center, \eqref{eq:extremal-ellipsoid-invariance} implies $T_\pi c=c$ for every $\pi\in\Pi$. Averaging these identities and using \Cref{lem:simplex-permutation-symmetries} gives
$c=\tfrac{1}{(n+1)!}\sum_{\pi\in\Pi}T_\pi c=0$.
Thus both $E_J$ and $E_L$ are centered at the origin.

It remains to determine their shape. Since $E$ is centered at the origin, we may write it as $E=\{x\in\mathbb{R}^n:x^\top Qx\leq 1\}$ with $Q\succ 0$. Because $T_\pi$ is orthogonal, \eqref{eq:extremal-ellipsoid-invariance} is equivalent to $T_\pi^\top Q T_\pi=Q$ for every $\pi\in\Pi$. 
Multiplying $T_\pi^\top Q T_\pi=Q$ from the left by $a_i^\top$ and from the right by $a_j$ yields
\begin{equation}\label{eq:commonidentity}
  a_{\pi(i)}^\top Q a_{\pi(j)}=a_i^\top Q a_j
\end{equation}
for all $i,j$ and every $\pi\in\Pi$. For any $i,k$, there exists a permutation $\pi\in\Pi$ such that $\pi(i)=k$. Taking $j=i$ in \eqref{eq:commonidentity} therefore shows that all quantities $a_i^\top Q a_i$ are equal. Denote their value by $d$. Similarly, for any $i\neq j$ and $k\neq \ell$, there exists a permutation $\pi\in\Pi$ such that $\pi(i)=k$ and $\pi(j)=\ell$. Equation \eqref{eq:commonidentity} then implies all quantities $a_i^\top Q a_j$ with $i\neq j$ are equal. Denote their common value by $b$.

Fixing $i$ and using $\sum_{j=1}^{n+1}a_j=0$ gives $0=a_i^\top Q\sum_{j=1}^{n+1}a_j=d+nb$, and hence $b=-\tfrac{d}{n}$. Together with $\|a_i\|_2=1$ and
$a_i^\top a_j=-\frac{1}{n}$ for $i\neq j$ from
\eqref{eq:regular-simplex-normalization}, we therefore have $a_i^\top Q a_j=d\,a_i^\top a_j$ for all $i,j$. Since $a_1,\ldots,a_n$ form a basis of $\mathbb{R}^n$, it follows that $Q=dI_n$. Positive definiteness of $Q$ gives $d>0$, so $E$ is a Euclidean ball centered at the origin.
\end{proof}

\begin{proof}[Proof of \Cref{lem:exact-extremal-ellipsoids}]
We first compute the radius of the MVCE.  By \Cref{lem:spherical-extremal-ellipsoids}, this ellipsoid
is $r_LB_2^n$.  Every $a_i$ has unit norm and lies on the boundary of $B_2^n$, whereas every $-\alpha a_i$ has norm $\alpha$ and belongs to $B_2^n$.  Hence $B_2^n$
is the smallest origin-centered ball containing $S_{n,\alpha}$, and therefore $r_L=1$.

It remains to compute the radius of the MVIE.  By
\Cref{lem:spherical-extremal-ellipsoids}, this ellipsoid is
$r_JB_2^n$, and hence
$
 r_J=\max\{r\ge0:rB_2^n\subseteq S_{n,\alpha}\}.
$
By the separating hyperplane theorem, the inclusion
 $rB_2^n\subseteq S_{n,\alpha}$ holds if and only if
\begin{equation*}
r=\max_{x\in rB_2^n}\langle u,x\rangle\le\max_{x\in S_{n,\alpha}}\langle u,x\rangle  \,,
\end{equation*}
for every unit vector $u$.  Let
$g(u):=\max_{x\in S_{n,\alpha}}\langle u,x\rangle$. Then
\begin{equation}\label{eq:def-r_J}
 r_J
 =\max\{r\ge0:rB_2^n\subseteq S_{n,\alpha}\}
 =\min_{\|u\|_2=1}g(u).\,
\end{equation}

To obtain an exact formula for $r_J$, we need to rewrite \eqref{eq:def-r_J}.  For a unit vector $u$, let $s \in \mathbb{R}^{n+1}$ denote the vector in which $s_i:=\langle u,a_i\rangle$ for each $i$.  Since
$S_{n,\alpha}=\mathop{\mathbf{conv}}\nolimits\{a_i,-\alpha a_i:i=1,\ldots,n+1\}$, the objective function becomes
$
g(u)=\max\left\{\max_i s_i,-\alpha\min_i s_i\right\}.
$
Moreover, by $\sum_i a_i=0$ and
\Cref{lemma:simplex-frame-identities}, $s$ needs to satisfy constraints 
\begin{equation}\label{eq:directional-moments}
 \sum_{i=1}^{n+1}s_i=0,
 \qquad
 \sum_{i=1}^{n+1}s_i^2=\frac{n+1}{n}.
\end{equation}
Conversely, it could be shown that every $s$ satisfying \eqref{eq:directional-moments} arises from a unit vector $u$. 
Suppose that $s \in \mathbb{R}^{n+1}$ satisfies
\eqref{eq:directional-moments}, and define
$u:=\tfrac{n}{n+1}\sum_{j=1}^{n+1}s_j a_j$.
For every $i$, using $\sum_j s_j=0$ and
\eqref{eq:regular-simplex-normalization}, we have
$\langle u,a_i\rangle
=\tfrac{n}{n+1}\bigl(s_i-\tfrac1n\sum_{j\neq i}s_j\bigr)
=s_i$.
Moreover,
$\|u\|_2^2
=\tfrac{n}{n+1}\sum_i s_i\langle a_i,u\rangle
=\tfrac{n}{n+1}\sum_i s_i^2=1$. 
Therefore, \eqref{eq:def-r_J} is equivalent to the following problem 
\begin{equation}\label{eq:rewrite-r_J}
\min \ \ \phi(s):=
\max\left\{\max_i s_i,-\alpha\min_i s_i\right\} \quad \operatorname{s.t.}\ 
e^\top s=0,\ 
\|s\|_2^2 =\tfrac{n+1}{n} 
\end{equation}
and $r_J$ is its optimal value.

Next we transform the problem \eqref{eq:rewrite-r_J} again
by replacing $s$ with $y:=\tfrac{s}{\phi(s)}$ (note that $\phi(s)>0$ for every feasible $s$).
Since $\phi$ is positively homogeneous,
$\phi(y)=1$, and since $e^\top s=0$, we also have $e^\top y=0$.
Moreover, $s=\phi(s)y$ and $\|s\|_2^2=\tfrac{n+1}{n}$ give
$\phi(s)=\sqrt{\tfrac{n+1}{n\|y\|_2^2}}$.
Conversely, if $y$ satisfies $e^\top y = 0$ and $\phi(y)=1$, then
$s:=\sqrt{\tfrac{n+1}{n \| y\|_2^2}}\,y$ satisfies
\eqref{eq:directional-moments}, and
$\phi(s)=\sqrt{\tfrac{n+1}{n \| y\|_2^2}}$.
Thus minimizing $\phi(s)$ in \eqref{eq:rewrite-r_J} is equivalent to
maximizing $\|y\|_2^2$ over the set defined by
$e^\top y=0$ and $\phi(y)=1$.

Finally, the constraint $\phi(y)=1$ can be relaxed to
$\phi(y)\leq 1$ without changing the maximum. Indeed, any nonzero
$y$ satisfying $\phi(y)<1$ can be scaled up until
$\phi(y)=1$, while preserving $e^\top y=0$ and strictly increasing
$\|y\|_2^2$. Thus every maximizer satisfies $\phi(y)=1$.
Furthermore, by the definition
$\phi(y)=\max\{\max_i y_i,-\alpha\min_i y_i\}$, the constraint
$\phi(y)\leq 1$ is equivalent to
$-\tfrac1\alpha\leq y_i\leq 1$ for every $i$. Therefore,
\begin{equation}\label{eq:john-radius-exact-maximization}
 r_J=\sqrt{\frac{n+1}{n}} \frac{1}{\|y^\star\|} \text{ where } y^\star \in \arg  
 \max
 \|y\|_2 \ \ \operatorname{s.t.} \
 \sum_{i=1}^{n+1}y_i=0,\ \text{and}\
 -\frac1\alpha\le y_i\le1\
 \text{ for all }i\,.
\end{equation}

It remains to evaluate the maximization problem in \eqref{eq:john-radius-exact-maximization}.  The feasible set in
\eqref{eq:john-radius-exact-maximization} is a compact polytope, and the objective is convex, so it has a
maximizer at an extreme point.  At such a point, at most one coordinate lies strictly between $-\tfrac1\alpha$ and
$1$.  
Let $k:=\lfloor\tfrac{n+1}{1+\alpha}\rfloor$ and $\delta:=\{\tfrac{n+1}{1+\alpha}\} = \tfrac{n+1}{1+\alpha} - k$. 
Here $\{x\}:=x-\lfloor x\rfloor$ denotes the fractional part of $x$.
A maximizer $y^\star$ therefore has
$k$ coordinates equal to $1$,
$n-k$ coordinates equal to $-\tfrac1\alpha$, and one remaining coordinate
$\tfrac{n-k}{\alpha}-k$ (by the zero-sum constraint).
Since $\tfrac{n+1}{1+\alpha}=k+\delta$ and $n-k=\alpha k+(1+\alpha)\delta-1$,  the remaining coordinate can be
written as $\tfrac{n-k}{\alpha}-k  = \tfrac{(1+\alpha)(k+\delta) - 1 -k} {\alpha} -k = \tfrac{(1+\alpha)\delta - 1}{\alpha}$.
Consequently,
\[
\begin{aligned}
\|y^\star\|_2^2
&=k+\tfrac{n-k}{\alpha^2}+\Big(\tfrac{n-k}{\alpha}-k\Big)^2 =k+\tfrac{\alpha k+(1+\alpha)\delta-1}{\alpha^2}+\Big(\tfrac{(1+\alpha)\delta - 1}{\alpha}\Big)^2 = \tfrac{n+1}{\alpha} - \tfrac{(1+\alpha)^2}{\alpha^2}\delta(1-\delta)
\end{aligned}
\]
where the last equality uses $n+1=(1+\alpha)(k+\delta)$.
Substituting the value of $\|y^\star\|$ into
\eqref{eq:john-radius-exact-maximization} yields
\begin{equation}\label{eq:john-radius-exact-formula}
 r_J
 =\sqrt{\frac{\alpha}{n}} 
 \left(
 1-\frac{(1+\alpha)^2}{\alpha(n+1)}
 \delta
 \left(1-\delta\right)
 \right)^{-\tfrac{1}{2}} \,.
\end{equation}
Note that here $\delta = \left\{\frac{n+1}{1+\alpha}\right\}$.
 
Our claimed bounds follow from this formula \eqref{eq:john-radius-exact-formula}.  Since $\delta \in [0,1)$, it is easy to show that
$r_J\ge \sqrt{\tfrac{\alpha}{n}}$.  For the reverse bound,  first suppose that $n\ge3$.  Since $\alpha\in[\tfrac1n,1]$ and the function
$t\mapsto \tfrac{t}{(1+t)^2}$ is increasing on $[0,1]$, we have
$
 \frac{(1+\alpha)^2}{\alpha(n+1)}
 \le
 \frac{(1+\tfrac1n)^2}{(\tfrac1n)(n+1)}
 =
 \frac{n+1}{n}.
$
Moreover, the product $\delta(1-\delta)$ is at most $\tfrac14$.  Hence the term $\frac{(1+\alpha)^2}{\alpha(n+1)}
 \delta
 \left(1-\delta\right)$ is at most
$
 \frac{n+1}{4n}\le\frac13.
$

When $n=2$, we have $\alpha\in[\tfrac12,1]$.  For
$\alpha>\tfrac12$,
$
\delta
=\left\{\frac{3}{1+\alpha}\right\}
=\frac{2-\alpha}{1+\alpha}.
$
Hence
$
\frac{(1+\alpha)^2}{3\alpha}\delta(1-\delta)
=
\frac{(2-\alpha)(2\alpha-1)}{3\alpha}
=
\frac13-\frac{2(1-\alpha)^2}{3\alpha}
\leq\frac13.
$
At $\alpha=\tfrac12$, the same bound holds since $\delta=0$.
Thus, for every $n\ge2$, the parenthesized factor in
\eqref{eq:john-radius-exact-formula} is at least $\tfrac23$, and therefore
$r_J\le\sqrt{\tfrac{3\alpha}{2n}}$.  When $n=1$, necessarily
$\alpha=1$ and $r_J=1$.

Finally, if $\tfrac{n+1}{1+\alpha}$ is an integer, then $\delta=0$ and $r_J=\sqrt{\tfrac\alpha n}$.  This  
completes the proof.
\end{proof}

\begin{proof}[Proof of \Cref{thm:sharpness-family}]
The symmetry identity is \Cref{lem:family-symmetry}.  Let ${E}_J=r_JB_2^n$ and ${E}_L=B_2^n$ denote the MVIE and MVCE of $S_{n,\alpha}$, respectively.  For any rounding
$E\subseteq S_{n,\alpha}-c\subseteq\rho E$, the extremal-volume properties of ${E}_J$ and ${E}_L$ give
\begin{equation*}
 \rho^n
 =\frac{\operatorname{vol}_n(\rho E)}{\operatorname{vol}_n(E)}
 \ge
 \frac{\operatorname{vol}_n({E}_L)}{\operatorname{vol}_n({E}_J)}
 =\frac{\operatorname{vol}_n(B_2^n)}{\operatorname{vol}_n(r_JB_2^n)}
 =\left(\frac1{r_J}\right)^n\,.
\end{equation*}
The uniform upper bound $r_J\le\sqrt{\tfrac{3\alpha}{2n}}$ in \Cref{lem:exact-extremal-ellipsoids} therefore proves
\eqref{eq:uniform-family-rounding-lower-bound}.  If $\tfrac{n+1}{1+\alpha}$ is an integer, the same lemma
gives $r_J=\sqrt{\tfrac\alpha n}$, which proves \eqref{eq:family-rounding-lower-bound}.
\end{proof}

\appendix
\section{Numerical details for \texorpdfstring{\cref{fig:four-roundings}}{Figure 1}}\label{app:numerical-example}

The polytope in \cref{fig:four-roundings} is
\begin{equation}
 K:=\mathop{\mathbf{conv}}\nolimits\{(0,0),(2,0),(2.407407,0.498538),(0,\tfrac32),(-0.740741,0.593567)\}.
\end{equation} 
Numerically, $\mathop{\mathbf{sym}}\nolimits(K)\approx.7072$, and hence
$\sqrt{\tfrac{2}{\mathop{\mathbf{sym}}\nolimits(K)}}\approx1.6816$, strictly smaller than the factor $2$ guaranteed by the MVCE and MVIE.

The center and shape in \cref{fig:four-roundings}(a) are obtained from the MVCE problem
\eqref{eq:free-loewner-primal}, and those in \cref{fig:four-roundings}(b) are obtained from the MVIE problem
\eqref{eq:free-john-primal}.  In \cref{fig:four-roundings}(a), the solid ellipsoid is the MVCE and the
dashed ellipsoid is its $\tfrac12$-scaled copy.  In \cref{fig:four-roundings}(b), the solid ellipsoid is the MVIE and the dashed ellipsoid is its $2$-scaled copy.  In both figures the two displayed ellipsoids are tangent to $K$.

For \cref{fig:four-roundings}(c), we translate a Minkowski center to the origin and solve
\eqref{eq:center-regularized-loewner} with $\lambda=\lambda_\alpha$.  The dashed
inner ellipsoid is $c^\star+\sqrt{\tfrac\alpha n}\,{E}_{Q^\star}$, and the solid outer ellipsoid is
$c^\star+\sqrt{1-\lambda_\alpha(c^\star)^\top Q^\star c^\star}\,{E}_{Q^\star}$.
These are the inner ellipsoid and the refined outer ellipsoid in \Cref{lm:rounding-w-representation}.  The rounding factor is
$\sqrt{\tfrac{n}{\alpha}\bigl(1-\lambda_\alpha(c^\star)^\top Q^\star c^\star\bigr)}\approx1.6679$.
The outer ellipsoid is tangent to $K$, whereas the inner ellipsoid has positive slack despite appearing tangent.

For \cref{fig:four-roundings}(d), we normalize the translated halfspace description as in \eqref{eq:john-h-representation} and solve
\eqref{eq:john-primal} with $\lambda=\lambda_\alpha$.  The solid inner ellipsoid and
the dashed outer ellipsoid are, respectively,
$c^\star+\left(\min_{i}\tfrac{1-h_i^\top c^\star}{\lVert X^\star h_i\rVert_2}\right){E}_{Q^\star}$ and
$c^\star+\sqrt{\tfrac n\alpha}\,{E}_{Q^\star}$,
where $Q^\star=(X^\star)^{-2}$.  These are the refined inner ellipsoid and the outer ellipsoid in
\Cref{lm:john-rounding}.  The rounding factor is
$\sqrt{\tfrac n\alpha}\left(\min_{i}\tfrac{1-h_i^\top c^\star}{\lVert X^\star h_i\rVert_2}\right)^{-1}\approx1.6720$.
The inner ellipsoid is tangent to $K$, whereas the outer ellipsoid has positive slack despite appearing tangent.
 
The optimization models are implemented in Julia/JuMP using the solver MOSEK.

\end{document}